%% file: SISC_ML_BCD.tex
\documentclass[10pt,final]{siamltex}
\usepackage{array}
\usepackage{graphicx}
\usepackage{color}
\usepackage[cmex10]{amsmath}
\usepackage{amsfonts}
\usepackage[boxed,noline]{algorithm2e}
\usepackage{multirow}
\usepackage{pdflscape}

\usepackage{hyperref}

\DeclareMathOperator*{\argmin}{arg\,min}
\DeclareMathOperator*{\argmax}{arg\,max}

\graphicspath{{images/}}
\DeclareGraphicsExtensions{.png}

\newcommand{\bfx}{\mathbf{x}}
\newcommand{\bfz}{\mathbf{z}}
\newcommand{\bft}{\mathbf{t}}
\newcommand{\bfv}{\mathbf{v}}
\newcommand{\bfy}{\mathbf{y}}

\newcommand{\bfg}{\mathbf{g}}
\newcommand{\bfw}{\mathbf{w}}

\newcommand{\sign}{\mbox{sign}}
\newcommand{\support}{\textnormal{supp}}

\newcommand{\tr}{\mbox{tr}}

\newcommand{\N}{{\mathcal N}}
\newcommand{\C}{{\mathcal C}}

\newcommand{\F}{{\mathcal F}}
\newcommand{\R}{{\mathcal R}}
\newcommand{\Sh}{\mathcal{S}}

\newcommand{\real}[1]{\ensuremath{\mathbb{R}^{\mathbf{#1}}}}

\newcommand{\mr}[2]{\multirow{#1}{*}{#2}}
\newcommand{\mc}[3]{\multicolumn{#1}{#2}{#3}}

\newcommand{\mydef}{\ensuremath{\stackrel{\triangle}{=}}}
\newcommand{\Active}{\F}

\newcommand{\ve}[1]{\ensuremath{\mathbf{#1}}}
\newcommand{\ves}[2]{\ensuremath{\mathbf{#1}_{#2}}}

\newcommand{\mat}[1]{\ensuremath{#1}}
\newcommand{\matt}[1]{\ensuremath{\tilde{\mat{#1}}}}

\newcommand{\normdist}{\ensuremath{\mathcal{N}}}
\newcommand{\gaussian}[3]{\ensuremath{ #1 \sim  \normdist \left(#2,#3\right)} }

\newcommand{\normone}[1]{\ensuremath{\|{#1}\|_1}}

\begin{document}
\title{A multilevel framework for sparse optimization with
application to inverse covariance estimation and logistic regression
\thanks{The research leading to these results has received funding from the European Union's -
 Seventh Framework Programme (FP7/2007-2013) under grant agreement no 623212---MC Multiscale Inversion. This research was also funded (in part) by the Intel Collaborative Research Institute for Computational Intelligence
(ICRI-CI).
}
}

\author{Eran Treister\thanks{Department of Earth and Ocean Sciences, University of British Columbia, Vancouver, Canada. {\tt eran@cs.technion.ac.il.}}
\and Javier S. Turek\thanks{Intel Labs, 2111 NE 25th Ave., Hillsboro, OR 97124,  {\tt javier.turek@intel.com.}} \and Irad Yavneh\thanks{Department of Computer Science, Technion---Israel Institute of Technology, Haifa, Israel. {\tt irad@cs.technion.ac.il.}}}
\maketitle

\begin{abstract}
%A multilevel framework is presented for solving $l_1$ regularized sparse optimization
%problems, which are common in the fields of computational biology, signal
%processing and machine learning. Such $l_1$ regularization is utilized to find sparse
%minimizers of convex functions, and is mostly known for its appearance in the LASSO problem,
%where the $l_1$ norm regularizes a quadratic function. To solve such problems efficiently,
%we take advantage of the expected sparseness of the solution. A multilevel hierarchy of problems of similar type is created and
%traversed back and forth in order to accelerate the optimization process. This framework is then applied for solving the sparse inverse covariance estimation problem, where the inverse of an unknown covariance matrix of a multivariate normal
%distribution is estimated, under the assumption that it is sparse. To this end, an $l_1$ regularized
%log-determinant optimization problem needs to be solved. This task is challenging
%especially for large-scale data sets, due to time and memory limitations. Our numerical
%experiments demonstrate the efficiency of the multilevel framework for both
%medium and large scale instances of this problem.
Solving $l_1$ regularized optimization problems is common in the fields of computational biology, signal
processing and machine learning. Such $l_1$ regularization is utilized to find sparse
minimizers of convex functions. A well-known example is the LASSO problem,
where the $l_1$ norm regularizes a quadratic function. A multilevel framework is presented for solving such $l_1$ regularized sparse optimization problems efficiently.
We take advantage of the expected sparseness of the solution, and create a hierarchy of problems of similar type, which is
traversed in order to accelerate the optimization process. This framework is applied for solving two problems: (1) the sparse
inverse covariance estimation problem, and (2) $l_1$-regularized logistic regression.
In the first problem, the inverse of an unknown covariance matrix of a multivariate normal
distribution is estimated, under the assumption that it is sparse. To this end, an $l_1$ regularized
log-determinant optimization problem needs to be solved. This task is challenging
especially for large-scale datasets, due to time and memory limitations. In the second problem, the $l_1$-regularization is added to the logistic regression classification objective to reduce overfitting to the data and obtain a sparse model.
Numerical experiments demonstrate the efficiency of the multilevel framework in accelerating existing iterative solvers for both of these problems.
\end{abstract}

\pagestyle{myheadings} \thispagestyle{plain} \markboth{Eran
Treister, Javier S. Turek and Irad Yavneh}{Multilevel Sparse Optimization}

\section{Introduction}

Sparse solutions of optimization problems are often sought in various fields such as signal processing,
machine learning, computational biology, and others \cite{zou2005regularization,sra2012optimization}. Particular applications include sparse modelling of signals \cite{elad2010sparse}, compressed sensing \cite{Donoho06,CW08}, speech recognition \cite{INVCOVSPEECH}, gene network analysis \cite{GENEEXPRESSIONANALYSIS}, and brain connectivity \cite{BRAINCONNECTIVITY}. To promote sparsity, an $l_1$ norm regularization term is often introduced, generally leading to convex optimization problems of the form
\begin{equation}\label{eq:FX}
\bfx^* = \argmin_{\bfx  \in \mathbb{R}^{n}} \; F(\bfx) = \argmin_{\bfx  \in \mathbb{R}^{n}} \; f(\bfx) + \lambda\|\bfx\|_1,
\end{equation}
where the function $f(\bfx)$ is smooth (continuously differentiable) and convex,
and $\lambda$ is a positive scalar parameter that balances between sparsity and adherence to minimizing $f(\bfx)$. A larger parameter $\lambda$ tends to produce a sparser minimizer $\bfx^*$, but also a higher value for $f(\bfx^*)$. Problem \eqref{eq:FX} is convex but non-smooth due to the regularizer, and traditional optimization methods
such as gradient descent tend to converge slowly. For this reason, in many cases special methods are developed for specific instances of \eqref{eq:FX}. A well-known special case is the LASSO problem \cite{LASSO}, where $f(\bfx)$ is a quadratic function,
\begin{equation}\label{eq:FX_QUAD}
\bfx^* = \argmin_{\bfx \in \mathbb{R}^{n}} \; \frac{1}{2}\bfx^T H\bfx + \bfx^T\bfg  + \lambda\|\bfx\|_1,
\end{equation}
with $\bfg\in\mathbb{R}^{n}$, and $H\in\mathbb{R}^{n\times n}$ a positive semi-definite matrix.
This problem can often be solved quite efficiently by so-called ``iterative shrinkage'' or ``iterative soft thresholding (IST)'' methods \cite{FN03,SNM03,LASSOWITHSSF,LASSOWITHPCD,BF07,SPARSA,WYGZ09,ELADZIBULEVSKY} and other methods \cite{GPSR,LASSOWITHCD,LASSOWITHCD2,BCD00,mista2014}.

As a rule, iterative methods that have been developed for solving the quadratic problem \eqref{eq:FX_QUAD} can be generalized and used to solve the general problem \eqref{eq:FX}.
This can be done by rather straightforward approaches:  at each iteration the smooth part $f(\bfx)$ in \eqref{eq:FX} is approximated by some quadratic function, whereas the non-smooth $l_1$ term remains intact. Then, a descent direction is computed by approximately solving the resulting $l_1$-regularized quadratic minimization problem of the form \eqref{eq:FX_QUAD} by iterative shrinkage methods. This approximation can be applied using only first order information (i.e., $\nabla f$), or second order information (both $\nabla f$ and $\nabla^2 f$). The former approach results in a gradient-descent-like method, while the latter approach, on which we will focus in this paper, is called the ``proximal Newton'' method \cite{PROXIMALNEWTONGENERAL}. We give a precise description in the next section.

In this work we propose a multilevel framework for accelerating existing solvers for problem \eqref{eq:FX}, based on the work of \cite{TY}
which introduced a similar framework for the LASSO problem.
In this framework, the convergence of existing iterative methods is accelerated using a
nested hierarchy of successively smaller versions of the problem. Exploiting the sparsity of the sought approximation,
the dimension of problem \eqref{eq:FX} is reduced by temporarily ignoring
ostensibly irrelevant unknowns, which remain zero for the duration of the multilevel iteration. That is, each reduced
problem is defined by \eqref{eq:FX}, restricted to a specially
chosen subset of variables. This yields a nested hierarchy of problems.
Subspace corrections are performed by applying iterative methods to each of the low dimensional problems in succession, with the aim of accelerating the optimization process. Under suitable conditions, this algorithm converges to a global minimizer of \eqref{eq:FX}.

In the second and third parts of the paper, we apply our framework to the solution of (1) the sparse inverse covariance estimation problem, and (2) the $l_1$-regularized logistic regression problem. Both of these have the form of \eqref{eq:FX}, and we focus on the first one in more details because it is significantly more complicated and challenging. In this problem, the inverse of the covariance matrix of a multivariate normal distribution is estimated from a relatively small set of samples, assuming that it is sparse.  Its estimation is performed by solving an $l_1$ regularized log-determinant optimization problem, on which we elaborate later in this paper. Many methods were recently developed for solving the covariance selection problem \cite{banerjee2006convex,banerjee2008model,d2008first,GLASSO,guillot2012iterative,DCQUIC,BIGQUIC,QUIC,mazumder2012exact,ORTHANT,BCDIC}, and a few of those \cite{QUIC,BIGQUIC,ORTHANT,BCDIC} involve a proximal Newton approach. In the present work we mostly focus on large-scale instances of this problem, which are required in fMRI \cite{BIGQUIC} and gene expression analysis \cite{GENEEXPRESSIONANALYSIS,honorio2013inverse} applications, for example. Such large scale problems are very challenging, primarily because of memory limitations, and the only two published methods that are capable of handling them are \cite{BIGQUIC,BCDIC}. In this paper we review the method of \cite{BCDIC} and accelerate it by our multilevel framework.
Moreover, we show that our framework is more efficient than other acceleration strategies for this problem. These ideas can be exploited for other large-scale instances of log-determinant sparse optimization problems such as \cite{witten2009covariance,ravishankar2013learning,huang2013optimal}. Following this, we briefly describe the $l_1$-regularized logistic regression problem, and accelerate the two methods \cite{CDN} and \cite{NEWGLMNET} using our multilevel framework.

\section{First and second order methods for $l_1$ regularized sparse optimization}\label{sec:relaxations_general}

As mentioned above, problem \eqref{eq:FX} can be solved by adapting iterated shrinkage methods for \eqref{eq:FX_QUAD}.
To achieve that, at iteration $k$ the smooth function $f$ in \eqref{eq:FX} is replaced by a quadratic approximation around the current iterate $\bfx^{(k)}$ to obtain a descent direction $\bfz^{(k)}$. More specifically, $\bfz^{(k)}$ is obtained by approximately solving
\begin{equation}\label{eq:F_tilde_general}
\begin{array}{rl}
\bfz^{(k)} =& \displaystyle{\argmin_{\bfz\in\mathbb{R}^{n}}\tilde F(\bfx^{(k)}+\bfz)}\\
= & \displaystyle{\argmin_{\bfz\in\mathbb{R}^{n}}{f(\bfx^{(k)}) + \langle \bfg^{(k)}, \bfz\rangle + \frac{1}{2}\langle \bfz,H^{(k)}\bfz\rangle + \lambda\|\bfx^{(k)}+\bfz\|_1}},
\end{array}
\end{equation}
where $\bfg^{(k)} = \nabla f(\bfx^{(k)})$, and $H^{(k)}$ is a positive definite matrix that is method specific. This problem is similar to \eqref{eq:FX_QUAD} and can be solved using the shrinkage methods mentioned earlier. The role of $H^{(k)}$ in \eqref{eq:F_tilde_general} is to either incorporate second order information ($H^{(k)} = \nabla^2f(x^{(k)})$, the Hessian of $f$) or mimic it in a simpler and cheaper way. If $H^{(k)} = \nabla^2f(x^{(k)})$ then this results in the ``proximal Newton'' method \cite{PROXIMALNEWTONGENERAL}, and in that case, if \eqref{eq:F_tilde_general}
is solved to sufficient accuracy, the proximal Newton method has a superlinear asymptotic convergence rate \cite{PROXIMALNEWTONGENERAL}.
Once $\bfz^{({k})}$ is found, the next iterate is obtained by
\begin{equation}\label{eq:Line-search}
\bfx^{(k+1)} = \bfx^{(k)} + \alpha \bfz^{({k})},
\end{equation}
where $\alpha>0$ is a scalar which may be chosen a priori or determined by a line-search procedure. For example, one may use the Armijo rule \cite{ARMIJO} or an exact linesearch if possible \cite{WYGZ09}.

A simpler approach for defining \eqref{eq:F_tilde_general} is to include only first order information, and to use a simpler $H^{(k)}$, often chosen to be a diagonal positive definite matrix, denoted $D^{(k)}$. In this case, the problem \eqref{eq:F_tilde_general}
has a closed-form solution
\begin{equation}\label{eq:IST_direction}
\bfz^{({k})} = \Sh_{\lambda(D^{(k)})^{-1}}\left(\bfx^{(k)} - (D^{(k)})^{-1} \bfg^{(k)} \right) - \bfx^{(k)},
\end{equation}
where
\begin{equation}\label{eq:shrinkage}
\Sh_\lambda(t) = \sign(t)\cdot\max(0,|t|-\lambda)
\end{equation}
is the ``soft shrinkage'' function, which reduces the absolute value of $t$ by $\lambda$, or sets it to zero if $\lambda > |t|$. This results in a first order shrinkage method, which is similar to gradient descent or quasi Newton methods, and can be seen as a generalization of existing shrinkage methods for \eqref{eq:FX_QUAD}. More specifically, a generalization of the separable surrogate functionals (SSF) method of \cite{LASSOWITHSSF} can be obtained by defining $D^{(k)}=cI$, where $I$ is the identity matrix and $c > \rho(\nabla^2f(x^{(k)}))$ is an upper bound on the spectral radius of the Hessian. Similarly, a generalization of the parallel coordinate descent (PCD) method \cite{LASSOWITHPCD} for \eqref{eq:FX} reads $D^{(k)} = diag(\nabla^2f(x^{(k)}))$ (the diagonal part of the Hessian). Such generalizations, which are mentioned in \cite{LASSOWITHPCD} for example, are quite straightforward. Moreover, in \cite{SPARSA} for example, the general problem \eqref{eq:FX} is addressed rather than the quadratic \eqref{eq:FX_QUAD}, which is the actual target of this work. We note that the convergence guarantee of such methods for \eqref{eq:FX} requires modest assumptions on the smooth function $f$ (i.e., the level set $\{z: f (\bfz) \leq f(\bfz_0)\}$ is compact, and the Hessian is bounded: $ ||\nabla^2f(\bfx)||< M $). Generally, such methods converge linearly, but can be accelerated by subspace methods like SESOP \cite{LASSOWITHPCD,ELADZIBULEVSKY} and non-linear conjugate gradients \cite{ELADZIBULEVSKY,TY}.

Although the shrinkage methods are generally more efficient than other traditional methods, they can be further accelerated. A main problem of these methods
is that, during the iterations, the iterates $\bfx^{(k)}$ may be denser than the final solution. In particular, if we start from a zero initial guess, we would
typically get several initial iterates $\bfx^{(k)}$ which are far less sparse than the final minimizer $\bfx^*$. The rest of the iterates will gradually become sparser, until the sparsity pattern converges to that of $\bfx^*$. The main drawback is that those initial iterations with the denser iterates may be significantly more expensive than the later iterations with the sparse iterates. For example, when solving \eqref{eq:FX_QUAD} a matrix-vector multiplication $H\bfx^{(k)}$ needs to be applied at each shrinkage iteration \eqref{eq:IST_direction}. If the matrix $H$ is given explicitly (i.e., not as a fast operator) and $\bfx^{(k)}$ is sparse, then computations can be saved by not multiplying the columns of $H$ that correspond to the zero entries of $\bfx^{(k)}$.

One of the most common and simplest ways to address this phenomenon is by a continuation procedure \cite{WYGZ09,SPARSA}. In this procedure a sequence of problems \eqref{eq:FX} that correspond to a sequence of decreasing regularization parameters $\lambda$ is solved. This sequence starts with a relatively large value of $\lambda$ which is gradually decreased, and at each stage the new initial guess is given by the approximate solution to \eqref{eq:FX} obtained with the previous, bigger, $\lambda$. Since a bigger $\lambda$ yields a sparser solution, this continuation procedure may decrease the number of non-zeros in the iterates $\bfx^{(k)}$ at the expense of applying additional iterations for the larger $\lambda$'s.

\section{A multilevel framework for $l_1$ regularized sparse optimization}
We begin this section by providing some definitions and motivation that will be useful in the remainder of the paper. Let
\begin{equation}\label{eq:supp}
\support{(\bfx)} = \{i:x_i\neq0\}
\end{equation}
denote the support of the vector $\bfx$: the set of indices of its non-zero elements. The key challenge of sparse optimization is to identify the best (small)
support for minimizing the objective $f(\bfx)$. Additionally, we must calculate the values of the non-zeros of this minimizer. In many cases, if we were initially
given the support of the minimizer, it would make the solution of \eqref{eq:FX} easier\footnote{In such cases where a support is known or assumed, the $l_1$ regularization may be dropped from the problem. However, here we focus on the solution of \eqref{eq:FX} for an unknown support, and use the known support case only as a motivation. Therefore, in this discussion we keep the $l_1$ regularization also if the support is known.}. For example, the LASSO problem \eqref{eq:FX_QUAD}
has $n$ variables, however, if the support of its minimizer $\C^* = \support{(\bfx^*)}$ is known and consists of only about $1\%$ of the $n$ entries, we can ignore all entries not in $\C^*$ and solve the problem
\begin{equation}\label{eq:FX_QUAD_supp}
\min_{\bfx_{c^*} \in \mathbb{R}^{|c^*|}} \; \frac{1}{2}\bfx_{c^*}^T H_{c^*}\bfx_{c^*} + \bfx_{c^*}^T\bfg_{c^*}  + \lambda\|\bfx_{c^*}\|_1,
\end{equation}
where $\bfx_{c^*}$, $H_{c^*}$ and $\bfg_{c^*}$ are the same components $\bfx$, $H$ and $\bfg$ from \eqref{eq:FX_QUAD}, restricted to the entries in $\C^*$. This is the same problem as \eqref{eq:FX_QUAD}, but is about 100 times smaller and therefore much cheaper to solve. This motivates our approach.

In our multilevel framework, the convergence of common iterative methods is accelerated using a
nested hierarchy of smaller versions of the problem referred to as \emph{coarse problems}.
At each multilevel iteration denoted by ``ML-cycle'',
%$$
%\bfx^{(k+1)} = \mbox{ML-cycle}(\bfx^{(k)}),
%$$
we define a hierarchy of such coarse problems using the sparsity of the iterated approximations $\bfx^{(k)}$.
Each coarse problem is defined by \eqref{eq:FX}, restricted to a subset of the variables, while keeping the other variables as zeros.
This process is repeated several times, yielding a nested hierarchy of problems.
In each ML-cycle we traverse the entire hierarchy of levels, from the coarsest to the finest, applying iterations using methods like \eqref{eq:IST_direction}
or proximal Newton over each of the coarse problems in turn. We henceforth refer to such iterations as \emph{relaxations}.
These aim to activate the variables that comprise the support of a minimizer. We iteratively
repeat these ML-cycles until some convergence criterion is satisfied.

\subsection{Definition of the coarse problems and multilevel cycle}

As noted, at each level $l$ we define a reduced coarse problem by limiting problem \eqref{eq:FX}
to a subset of entries, denoted by $\C_l \subset\{1,...,n\}$. In this subsection we assume that the subset $\C_l$ is given and defer the discussion on how it is chosen to the next section.
Given $\C_l$, the coarse problem for level $l$ is defined by
\begin{equation}\label{eq:FX_coarse}
\min_{\substack{\bfx  \in \mathbb{R}^{n}, \\\support{(\bfx)}\subseteq\C_l}} \; F(\bfx),
\end{equation}
where $F$ is the same objective as in \eqref{eq:FX} and $supp{(\bfx)}$ is defined in \eqref{eq:supp}. Effectively, \eqref{eq:FX_coarse} has
only $|\C_l|$ unknowns, hence it is of lower dimension. Furthermore, if $\C_l$ contains the support of a minimizer of \eqref{eq:FX}, i.e., $\C_l \supseteq \support{(\bfx^*)}$, then $\bfx^*$ is also a solution of \eqref{eq:FX_coarse}. Otherwise, the solutions of the two problems are not identical.

We define our multilevel hierarchy by choosing nested subsets of variables $\{\C_l\}_{l=0}^{L}$. Given the current iterate, $\bfx^{(k)}$ we define the hierarchy
\begin{equation}\label{eq:hierarchy1}
 \{1,...,n\} = \C_0 \supset \C_1\supset ...\supset \C_L = \support{(\bfx^{(k)})}.
\end{equation}
In the ML-cycle, we treat the levels from $L$ to $0$ in succession by applying relaxations for the reduced problem \eqref{eq:FX_coarse} corresponding to each subset $\C_l$. We typically apply only one or two relaxations on levels $L-1,...,0$. On the coarsest level $L$, more relaxations may be applied because $\C_L$ is typically small and they are not expensive. The multilevel cycle procedure is presented in detail in Algorithm \ref{alg:ML}.

To define an iterative relaxation for \eqref{eq:FX_coarse} at each of the levels in \eqref{eq:hierarchy1}, one can adapt most iterative relaxations
that are suitable for the finest problem \eqref{eq:FX} by allowing only the elements in $\C_l$ to vary and fixing the other elements to zero. It is important to choose a relaxation whose cost is at worst proportional to the number of unknowns $|\C_l|$. As an example, consider again the relation between the original LASSO problem \eqref{eq:FX_QUAD} and one restricted to a much smaller set of variables, $|\C_l| \ll n$. If the matrices are given explicitly in memory, then each shrinkage iteration for \eqref{eq:FX_QUAD_supp} is significantly cheaper than a shrinkage iteration for \eqref{eq:FX_QUAD}, roughly proportional to $|\C_l|$. We note that any specific problem of the form \eqref{eq:FX}, for which there exists some relaxation whose cost is proportional to $|\C_l|$ when applied to the restricted problem \eqref{eq:FX_coarse}, is suitable to be accelerated with our ML approach.

\begin{algorithm}
\small
\DontPrintSemicolon
\label{alg:ML} \caption{A multilevel cycle for sparse optimization.}
\KwSty{Algorithm: $\bfx^{(k+1)} \leftarrow$ ML-cycle$(\bfx^{(k)})$}\;
\emph{\% Parameters: }\;
\emph{\% $\bfx^{(k+1)} \leftarrow$ Relax$(\bfx^{(k)},\C)$ : a relaxation method for \eqref{eq:FX_coarse}.}\;
\emph{\% Number of relaxations at each level: $\nu$.}\;
\emph{\% Maximal number of relaxations on the coarsest level: $\nu_{c}$.}\;
\begin{enumerate}\Indm
\item Define the hierarchy $\{\C_l\}_{l=0}^L$ in \eqref{eq:hierarchy1}. \;
\item Set $\bfx\leftarrow \bfx^{(k)}$\;
\item \label{step:coarsest}Apply $\bfx\leftarrow Relax(\bfx,\C_L)$ until coarsest-level convergence criterion is satisfied \;
\item \textbf{For }$l=L-1,...,0$ \;
$\quad\quad$ Apply $\bfx\leftarrow Relax(\bfx,\C_l)$ $\nu$ times for \eqref{eq:FX_coarse} restricted to $\C_l$.\;
\textbf{end}\;
\item Set $\bfx^{(k+1)} \leftarrow \bfx$\;
\end{enumerate}
\end{algorithm}

\subsection{Choosing the coarse variables} \label{sec:likely}
Given $\C_l$, our task is to select a subset of indices, $\C_{l+1}$, that is significantly smaller than $\C_l$ (see below), and is deemed most likely to contain the support of the ultimate solution we are seeking, $\support{(\bfx^*)}$. We begin by choosing to include the indices that are in the support of the current iterate, $\support{(\bfx^{(k)})}$. To these we add indices not in $\support{(\bfx^{(k)})}$, that are estimated to be relatively likely to end up in $\support{(\bfx^*)}$, namely, indices corresponding to variables $i$ with a relatively large absolute value of the current gradient,
$|(\nabla f(\bfx^{(k)}))_i|$. To motivate this choice, consider the minimization problem \eqref{eq:F_tilde_general} restricted to a single element of $\bfz$, denoted $z_i$, and assume that $\bfx^{(k)}_i = 0$. We get the scalar minimization problem
\begin{equation}\label{eq:1d_l1}
z_i^{opt} = \argmin_{z_i}\;\;\frac{1}{2}az_i^2 + bz_i + c + \lambda|z_i|,
\end{equation}
where $a = H^{(k)}_{ii} > 0$, $b = \bfg^{(k)}_i = (\nabla f(\bfx^{(k)}))_i$, and $c$ is a constant. A closed-form solution of \eqref{eq:1d_l1}, is given by
\begin{equation}
z_i^{opt} = \left\{
\begin{array}{ll}
\frac{b - \lambda}{a} \quad if \quad b > \lambda \\
\frac{b + \lambda}{a}  \quad if \quad b < -\lambda\\
0 \quad otherwise,
\end{array}
\right.
\end{equation}
which corresponds to element $i$ in \eqref{eq:IST_direction}. This indicates that if $|b| = |(\nabla f(\bfx^{(k)}))_i|$ is relatively large, then we
have a relatively good chance that the variable $z_i$ will become non-zero and $i$ will enter the support in the next iterate $\bfx^{(k+1)}$.

To summarize, for a given $\C_l$ we first decide on the size of $\C_{l+1}$. In this work we choose $|\C_{l+1}| = \max \left( \lceil\frac{1}{2}|\C_l|\rceil, |\support(\bfx^{(k)})| \right)$, and terminate the coarsening (setting $L = l$) when $|\C_l| = |\support(\bfx^{(k)})|$. Then, to populate $\C_{l+1}$, we first choose to include $\support{(\bfx^{(k)})}$, and then add the indices of the $|\C_{l+1}| - |\support{(\bfx^{(k)})}|$ additional variables $i$ with the largest values of $|(\nabla f(\bfx^{(k)}))_i|$. The choice of the coarsening ratio of approximately $1/2$ turns out to strike a good balance between ML-cycle cost and efficacy. If the cost of the relaxation at level $l$ is proportional to $|\C_l|$, then the cost of a ML-cycle with $\nu$ relaxations per level is approximately equal to  $2\nu$ relaxations on the finest level. This means that, although we include a relatively large fraction of the variables when we coarsen to the next level, the cost of the entire cycle remains relatively small.

Finally, we note that in practice we define the nested hierarchy $\{\C_l\}_{l=0}^L$ using the gradient from the relaxation on the finest level of the previous cycle, which includes all the variables. This relaxation is also used for monitoring convergence, as it is the only place where the gradient is calculated for all variables.

\section{Theoretical results} \label{sec:theory}

In this section we state some theoretical observations regarding the relaxation methods defined by \eqref{eq:F_tilde_general}-\eqref{eq:Line-search}, and our multilevel framework in Algorithm \ref{alg:ML}.

\subsection{Theoretical results for the relaxation methods}
We show that under suitable conditions \emph{any} relaxation method defined by \eqref{eq:F_tilde_general}-\eqref{eq:Line-search} is monotonically decreasing and convergent. We first prove the following lemmas.

\begin{lemma} {\rm (Monotonicity of the relaxation.)} \label{lem:monotonicity}
Assume that the Hessian is bounded $||\nabla^2f(\bfx)||< M$, and that $\bfx^{(k+1)} = Relax(\bfx^{(k)})$ is defined by \eqref{eq:F_tilde_general}-\eqref{eq:Line-search}, with $\bfg^{(k)}=\nabla f(\bfx^{(k)})$ and $H^{(k)}\succeq\gamma_{min}I\succ0$, where $\gamma_{min}$ is a positive constant. Then
\begin{equation}\label{eq:MLMI2}
F(\bfx^{(k)}) - F(\bfx^{(k+1)}) \geq K\cdot\|\bfx^{(k)} - \bfx^{(k+1)}\|^2 ~~~~ \forall \bfx^{(k)} \in \mathbb{R}^{n},
\end{equation}
where $K$ is a positive constant. Furthermore, the linesearch parameter $\alpha$ in \eqref{eq:Line-search} can be chosen to be
bounded away from zero, i.e., $\alpha \geq \alpha_{min}>0$.
\end{lemma}
\begin{proof}
The following analysis is inspired by \cite{QUIC} and \cite{SPARSA}. Let us drop the superscript
$^{(k)}$, and write for any $\bfx$, $\bfz$ and $0<\alpha<1$
\begin{equation}\label{eq:triangle_inequality}
\|\bfx+\alpha\bfz\|_1 = \|\alpha(\bfx+\bfz) + (1-\alpha)\bfx\|_1 \leq \alpha\|\bfx+\bfz\|_1 + (1-\alpha)\|\bfx\|_1.
\end{equation}
Next, since the Hessian is bounded we can write for any $\bfx$, $\bfz$ and $\alpha$
\begin{equation}\label{eq:bounded_Hessian}
f(\bfx+\alpha\bfz) \leq f(\bfx) + \alpha\bfz^T\nabla f + \frac{1}{2}\alpha^2 M\|\bfz\|^2.
\end{equation}
Now, let us assume that $\bfz$ was yielded by \eqref{eq:F_tilde_general}. We obtain
\begin{equation}\label{eq:alpha_proof}
\begin{array}{rl}
F(\bfx) - F(\bfx+\alpha\bfz) = & f(\bfx) + \lambda\|\bfx\|_1 - f(\bfx+\alpha\bfz) - \lambda\|\bfx+\alpha\bfz\|_1\\
\geq &  \lambda\|\bfx\|_1-\left(\alpha\bfz^T\nabla f + \frac{1}{2}\alpha^2 M\|\bfz\|^2 + \lambda\|\bfx+\alpha\bfz\|_1\right) \\
\geq &  -\left(\alpha\bfz^T\nabla f + \frac{1}{2}\alpha^2 M\|\bfz\|^2 + \lambda\alpha\|\bfx+\bfz\|_1- \lambda\alpha\|\bfx\|_1\right) \\
 =  & -\alpha\left(\bfz^T\nabla f + \frac{1}{2}\bfz^TH\bfz + \lambda\|\bfx+\bfz\|_1 - \lambda\|\bfx\|_1\right)  \\
    & + \frac{1}{2}\alpha\bfz^TH\bfz - \frac{1}{2}\alpha^2 M\|\bfz\|^2\\
\geq & \frac{1}{2}\alpha\bfz^TH\bfz - \frac{1}{2}\alpha^2 M\|\bfz\|^2,\\
\end{array}
\end{equation}
where the first and second inequalities are obtained by \eqref{eq:bounded_Hessian} and \eqref{eq:triangle_inequality}, respectively, and the third inequality follows from the fact that $\bfz$ achieves a better objective in \eqref{eq:F_tilde_general} than 0. Now, because we assume that $H\succeq\gamma_{min}I$ is used in the relaxation, then following \eqref{eq:alpha_proof} we write
\begin{equation}\label{eq:alpha_proof2}
\begin{array}{rl}
F(\bfx) - F(\bfx+\alpha\bfz) \geq & \frac{1}{2}( \gamma_{min} - \alpha M)\alpha\|\bfz\|^2 = K\cdot\alpha\|\bfz\|, \\
\end{array}
\end{equation}
which is always positive for any $0<\alpha < \frac{\gamma_{min}}{M}$. This proves \eqref{eq:MLMI2} for $K = \frac{1}{2}( \gamma_{min} - \alpha M)$.
%\tred{Furthermore, any linesearch procedure which maximizes the left hand side of \eqref{eq:alpha_proof2}, can always choose a value $\alpha>0$ that reduces the objective, so there exists a lower bound on the linesearch parameter $\alpha$ throughout the iterations.} %In particular, the lower bound in \eqref{eq:alpha_proof2} is maximized for $\alpha = \frac{\gamma_{min}}{2M}$.
\end{proof}

\bigskip
For the following results we use the notion of sub-gradients. $\partial F(\bfx)$, the sub-differential of $F$, is the set
\begin{equation} \label{eq:subgradient1}
\partial F(\bfx) = \left\{\nabla f(\bfx) + \lambda\bft : \begin{array}{lc}
t_i = \sign(x_i) ~~~~\mbox{if $x_i \neq 0$}  \\
t_i \in [-1,1] ~~~~~~ \mbox{if $x_i = 0$}
\end{array}\right\}.
\end{equation}
\noindent A vector $\bfx^*$ is a minimizer of \eqref{eq:FX}, if
and only if $0 \in \partial F(\bfx^*)$ \cite{fletcher1987practical}. We now extend Lemma 2 in \cite{SPARSA} to any relaxation of type \eqref{eq:F_tilde_general}-\eqref{eq:Line-search}. This lemma shows that, under suitable conditions, any point $\bar\bfx$ is either a stationary point of $F(\cdot)$, or else the result of $Relax(\bar\bfx)$ is a substantial distance away from $\bar\bfx$. The proof for this lemma is similar to the proof in \cite{SPARSA}.

\begin{lemma}{\rm (No stagnation of the relaxation.)}\label{lem:sufficient_reduction}
Let $\{\bfx^{(k)}\}$ be a series of points produced by $\bfx^{(k+1)} =
\mbox{Relax}(\bfx^{(k)})$, defined by \eqref{eq:F_tilde_general}-\eqref{eq:Line-search} with $ \gamma_{max}I \succeq H^{(k)}\succeq\gamma_{min}I\succ0$.
Let $\{\bfx^{(k_j)}\}$ be any infinite and converging subseries of $\{\bfx^{(k)}\}$, and let $\bar\bfx$ denote its limit.
Then $\bar\bfx$ is a stationary point of $F(\cdot)$ in \eqref{eq:FX}.
\end{lemma}
\begin{proof}
%Assume by contradiction that $\bar{\bfx}$ is not a stationary point of $F$.
Since the subseries $\{\bfx^{(k_j)}\}$ converges to $\bar\bfx$, then $\{F(\bfx^{(k_j)})\}$ converges to $F(\bar\bfx)$. Following Lemma \ref{lem:monotonicity}, the full series $\{F(\bfx^{(k)})\}$ is monotone and hence convergent because $\{F(\bfx^{(k_j)})\}$ is convergent. Therefore $\{F(\bfx^{(k_j+1)})-F(\bfx^{(k_j)})\}\rightarrow 0$, which implies following \eqref{eq:alpha_proof2} that $\|\bfx^{k_j} - \bfx^{k_j+1}\| \rightarrow 0$, and $\lim_{j\rightarrow\infty}\bfx^{(k_j+1)} = \bar\bfx$. By \eqref{eq:F_tilde_general}, the condition
\begin{equation} \label{eq:subgrad}
0 \in \partial\tilde F = \{\nabla f(\bfx^{(k_j)}) + H^{(k_j)}\bfz^{(k_j)} + \lambda\bft\}
\end{equation}
is satisfied for $\bfz^{(k_j)}$, where $\bft = \partial \|\bfx^{(k_j)} + \bfz^{(k_j)}\|_1$.
Then, by \eqref{eq:Line-search} we have $\bfx^{(k_j+1)} = \bfx^{(k_j)} + \alpha\bfz^{(k_j)}$. Because $\alpha > \alpha_{min}$ by Lemma \ref{lem:monotonicity}, $\|\bfx^{k_j} - \bfx^{k_{j}+1}\| \rightarrow 0$ implies $\|\bfz^{(k_j)}\| \rightarrow 0$, which, due to the upper bound on $H^{(k_j)}$, leads to $\|H^{(k_j)}\bfz^{(k_j)}\| \rightarrow 0$. Now as in \cite{SPARSA}, by taking the limit as $j\rightarrow\infty$, and using outer semicontinuity of $\partial \|\cdot\|_1$, we have that $0\in\partial F (\bar\bfx)$ defined in \eqref{eq:subgradient1}, hence $\bar\bfx$ is a stationary point of $F$.
%which contradicts the non-stationarity of $F$
\end{proof}

\bigskip
Next, we show that relaxation of type \eqref{eq:F_tilde_general}-\eqref{eq:Line-search} converges to a stationary point of \eqref{eq:FX}, which is also a minimum because $F(\bfx)$ is convex. Our proof follows the convergence proofs in \cite{LASSOWITHPCD,TY}.
\begin{theorem} {\rm (Convergence of the relaxation.)}\label{prop:Convergence2}
Assume that the level set $\R=\{\bfx:F(\bfx)\leq F(\bfx^{(0)})\}$ is
compact, and the Hessian is bounded: $ ||\nabla^2f(\bfx)||< M $. Let $\{\bfx^{(k)}\}$ be a
series of points produced by $\bfx^{(k+1)} =
\mbox{Relax}(\bfx^{(k)})$, defined by \eqref{eq:F_tilde_general}-\eqref{eq:Line-search}, with $\gamma_{min}I\prec H^{(k)} \prec\gamma_{max}I$, starting from an initial guess
$\bfx^{(0)}$. Then any limit point $\bfx^*$ of the sequence
$\{\bfx^{(k)}\}$ is a stationary point of $F$ in \eqref{eq:FX},
i.e., $0\in\partial F(\bfx^*)$, and $F(\bfx^{(k)})$ converges to
$F(\bfx^*)$.
\end{theorem}
\begin{proof}
By Lemma \ref{lem:monotonicity}, the series $\{F(\bfx^{(k)})\}$ is
monotonically decreasing. Since the objective $F$ in \eqref{eq:FX} is non-negative,
it is bounded from below, and hence the series $\{F(\bfx^{(k)})\}$
converges to a limit. Because the level set $\R$ is compact by
assumption, we have that $\{\bfx^{(k)}\}$ is bounded in $\R$, and
therefore there exists a sub-series $\{\bfx^{k_n}\}$ converging to
a limit point $\bfx^*$. By Lemma \ref{lem:sufficient_reduction}, the point $\bfx^*$ is a stationary point of $F(\cdot)$.
%\tred{Now let us assume to the contrary that $\bfx^*$ is not a stationary point}.
%By Lemma \ref{lem:sufficient_reduction}, since $\bfx^{(k_n)}\rightarrow \bfx^*$, there exists $\epsilon>0$ such that $\|\bfx^{(k_n)} - \bfx^{(k_{n}+1)}\| > \epsilon$ and there are infinitely
%many $k_n$'s satisfying that. By Lemma \ref{lem:monotonicity} we have
%infinitely many $k_{n}$'s satisfying
%\begin{equation}
%F(\bfx^{(k_n)}) - F(\bfx^{(k_{n}+1)}) \geq K\cdot\epsilon^2,
%\end{equation}
%which contradicts the fact that $F$ is bounded from below. This
%shows that the point $\bfx^*$ is stationary.
Since $F(\cdot)$ is
continuous, $\bfx^{(k_n)}\longrightarrow \bfx^*$ yields
$F(\bfx^{(k_n)})\longrightarrow F(\bfx^*)$. The limit of
$\{F(\bfx^{(k)})\}$ equals to that of any of its sub-series,
specifically $\{F(\bfx^{(k_n)})\}$, and thus
$F(\bfx^{(k)})\longrightarrow F(\bfx^*)$.
\end{proof}
\bigskip

The results above are intriguing because they show that we can use any positive definite $H^{(k)}$ in \eqref{eq:F_tilde_general}, and the resulting method converges. In particular, the analysis shows that one can use a positive definite inexact Hessian as in \cite{BCDIC,BIGQUIC}, and the method still converges.   In a way, this is similar to the property of preconditioners when solving linear systems. Now, one may wonder if it is possible to generate preconditioners $H^{(k)}$ for $\nabla^2f$ which are ``easily invertible'' in the sense of minimizing \eqref{eq:F_tilde_general}, and solve \eqref{eq:FX} more efficiently this way.

\subsection{Theoretical results for the multilevel framework}

From the definitions of \eqref{eq:FX} and \eqref{eq:FX_coarse}, we know that reducing $F(\bfx)$ on any of the coarser levels also reduces $F(\bfx)$ for the fine level. Therefore, if Algorithm 1 is used with a monotonically decreasing relaxation, then it is also monotonically decreasing. In addition, we have the following properties.
\begin{lemma} {\rm (Coarse solution correspondence.)} \label{prop:inter_solution}
Let $\C_l\supseteq \support(\bfx^*)$ be a subset of the variables \{1,...,n\}, where $\bfx^*$ is a solution of \eqref{eq:FX}. Let $\hat\bfx$ be a solution of problem \eqref{eq:FX_coarse} restricted to $\C_l$. Then $\hat\bfx$ is also a solution of \eqref{eq:FX}.
\end{lemma}
\begin{proof}
Because $\C_l\supseteq \support(\bfx^*)$, then $\bfx^*$ is a feasible point of \eqref{eq:FX_coarse}. Since $\hat\bfx$ is a solution of \eqref{eq:FX_coarse}, then $F(\hat\bfx)\leq F(\bfx^*)$. Therefore, $\hat\bfx$ is also a solution of \eqref{eq:FX}, because otherwise we contradict the optimality of $\bfx^*$.
\end{proof}

\bigskip

\noindent For the next two properties we assume that the coarsest problem is solved exactly in Algorithm \ref{alg:ML}. From lemma \ref{prop:inter_solution}, the following corollary immediately holds.

\begin{corollary}\label{cor:multilevel_direct_solution}
If $\C_L \supseteq \support(\bfx^*)$ at the $k$-th cycle of Algorithm \ref{alg:ML}, then problem \eqref{eq:FX} is solved at that cycle.
\end{corollary}
%\begin{corollary}
%$\bfx^*$ is a stationary point of Algorithm \ref{alg:ML} if it is a stationary point of the relaxation method $Relax()$.
%\end{corollary}
\bigskip

\begin{theorem} {\rm (No Stagnation of ML-cycle.)}\label{prop:NoStagnation}
Assume that the conditions of Lemma \ref{lem:monotonicity} hold for the relaxation method used in Algorithm \ref{alg:ML}. Let $\bfx$ be the solution of the coarsest level problem at Step \ref{step:coarsest} of Algorithm \ref{alg:ML}. If $\C_L \not\supseteq \supp(\bfx^*)$, then at least one
iterated shrinkage relaxation on one of the levels $L-1,...,0$ must change $\supp(\bfx)$.
\end{theorem}

\begin{proof}
Because $\bfx$ is a minimizer of the coarsest problem \eqref{eq:FX_coarse},
then for all $j \in \C_L$
\begin{equation}\label{eq:coarseOptimality}
\begin{array}{lcr}
(\nabla f(\bfx))_j + \lambda\sign(x_j) = 0 &\mbox{if }x_j \neq 0,  \\
|(\nabla f(\bfx))_j| \leq \lambda & \mbox{if } x_j = 0.
\end{array}
\end{equation}
Now, since $\C_L \not\supseteq \supp(\bfx^*)$, $\bfx$ is not a minimizer of the unrestricted problem \eqref{eq:FX}, so $0 \not\in \partial F(\bfx)$. Therefore, there exists
at least one variable $q\not\in\C_L$ for which $|(\nabla f(\bfx))_q| > \lambda$. Suppose that \eqref{eq:coarseOptimality} holds for $\{\C_l\}_{l=\hat{l}-1}^L$, such that $\C_{\hat{l}}$ is the coarsest level in the multilevel hierarchy that includes such a variable. Because $\{\C_l\}_{l=\hat{l}-1}^L$ satisfy \eqref{eq:coarseOptimality}, $\bfx$ is a stationary point of all the relaxations \eqref{eq:IST_direction} on those levels. However, on level $\hat{l}$ \eqref{eq:coarseOptimality} is violated, and the relaxation yields a direction $\bfz_{c_{\hat{l}}}$ fulfiling \begin{equation} \label{eq:subgrad2}
0 \in (\partial\tilde F)_{c_{\hat{l}}} = \{(\nabla f(\bfx))_{c_{\hat{l}}} + H_{c_{\hat{l}}}\bfz_{c_{\hat{l}}} + \lambda\bft_{c_{\hat{l}}}\},
\end{equation}
where $\bft_{c_{\hat{l}}} = \partial \|\bfx_{c_{\hat{l}}} + \bfz_{c_{\hat{l}}}\|_1$, and $(\nabla f(\bfx))_{c_{\hat{l}}}$ and $(H)_{c_{\hat{l}}}$ are the gradient and the Hessian approximation of the relaxation restricted to the entries in $\C_{\hat{l}}$. Since $\bfz_{c_{\hat{l}}}\neq0$ and the linesearch parameter $\alpha>\alpha_{min}$, then $\|\bfx-Relax(\bfx)\| > 0$, and hence, following Lemma \ref{lem:monotonicity}, $F(Relax(\bfx)) < F(\bfx)$. Now, if the support of $\bfx$ did not change following this relaxation, i.e., $\support{(Relax(\bfx))} = \support{(\bfx)}$, this would contradict the optimality of $\bfx$ with respect to the levels $\{\C_l\}_{l=\hat{l}-1}^L$.
\end{proof}
\bigskip

Our last Theorem proves that Algorithm \ref{alg:ML} converges when used with a suitable relaxation method. The Algorithm falls into the block coordinate descend framework in \cite{BCDTsengYun} where the blocks are the sets $\{\C_l\}_{l=0}^L$ in all levels. In particular, since $\C_0=\{1,...,n\}$, then all multilevel cycles end with a relaxation that includes all the variables, and hence the Gauss-Seidel rule in \cite{BCDTsengYun} is satisfied at most every $L\cdot\nu + \nu_c$ relaxations.

\begin{theorem} {\rm (Convergence of Algorithm \ref{alg:ML}.)}\label{prop:Convergence3}
Assume that the conditions of Theorem \ref{prop:Convergence2} hold for $f(\bfx)$ and $Relax(\bfx)$. Let $\{\bfx^{(k)}\}$ be a
series of points produced by $\bfx^{(k+1)} =\mbox{ML-cycle}(\bfx^{(k)})$, defined by Algorithm \ref{alg:ML} with $\nu>0$ and $\nu_c>0$, starting from an initial guess $\bfx^{(0)}$. Then any limit point $\bfx^*$ of the sequence
$\{\bfx^{(k)}\}$ is a stationary point of $F$ in \eqref{eq:FX}, and $F(\bfx^{(k)})$ converges to
$F(\bfx^*)$.
\end{theorem}
\begin{proof}
Let us now define $\{\bfy^{(s)}\}$ to be the series of points generated by all the relaxations that are performed within the cycles for producing $\{\bfx^{(k)}\}$. Lemma \ref{lem:monotonicity} and the relation between the problems \eqref{eq:FX} and \eqref{eq:FX_coarse} imply that $\{F(\bfy^{(s)})\}$ is monotonically non-increasing. Similarly to the proof of Theorem \ref{prop:Convergence2}, since $F(\cdot)$ in \eqref{eq:FX} is bounded from below, the series $\{F(\bfy^{(s)})\}$ converges to a limit, and therefore there exists a sub-series $\{\bfy^{(s_n)}\}$ converging to a limit point $\bfx^*$. Because we apply the same type of relaxation on all levels then following \eqref{eq:alpha_proof2}, $\{F(\bfy^{(s_j)})-F(\bfy^{(s_j+1)})\}\rightarrow 0$ implies that $\|\bfy^{(s_j)} - \bfy^{(s_j+1)}\| \rightarrow 0$, and $\lim_{j\rightarrow\infty}\bfy^{(s_j+1)} = \bfx^*$. In a similar way this leads to $\lim_{j\rightarrow\infty}\bfy^{(s_j+t)} = \bfx^*$ for $t=1,...,L\cdot\nu + \nu_c$. By the definition of Algorithm \ref{alg:ML}, a full relaxation is applied on one of the points $\bfy^{(s_j+t)}$ (i.e., a relaxation that includes all variables in $\C_0=\{1,...,n\}$). This means that at least one of the subseries $\bfy^{(s_j+t)}$ includes an infinite subseries of fine-level points converging to $\bfx^*$, and each direction obtained by the corresponding relaxation satisfies \eqref{eq:subgrad}. By the same arguments that follow Equation \eqref{eq:subgrad} in the proof of Lemma \ref{lem:sufficient_reduction}, $\bfx^*$ is a stationary point of $F(\cdot)$. Since $F(\cdot)$ is continuous, $\bfy^{(s_j)}\longrightarrow \bfx^*$ yields $F(\bfy^{(s_j)})\longrightarrow F(\bfx^*)$. The limit of
$\{F(\bfy^{(s)})\}$ equals to that of any of its sub-series,
specifically $\{F(\bfy^{(s_n)})\}$. Thus $F(\bfx^{(k)})$ which is a subseries of $\{F(\bfy^{(s)})\}$ converges to $F(\bfx^*)$.
\end{proof}
\bigskip

\textbf{Organization and notation.} Until now we described a general framework for $l_1$ regularized convex optimization. The remainder of the paper is devoted to two specific problems: one is the sparse inverse covariance estimation on which we focus extensively, and the other is $l_1$-logistic regression. For the first problem it is natural to consider the unknowns as a matrix (the estimated inverse of the covariance matrix), and hence we revert to the familiar matrix notation $A\in\mathbb{R}^{n\times n}$, instead of $\bfx\in\mathbb{R}^n$. This is the only difference in notation between the first and second parts (e.g., we minimize $F(\bfx)$ in sections 1-4, $F(A)$ in sections 5-8, and $F(\bfw)$ in sections 9-10. In all cases, $F()$ is the $l_1$-regularized non-smooth objective).

\begin{table*}[h]
\centering
\begin{tabular}{ccc}
\hline
Sections 1-4:  & Sections 5-8:  & Section 9-10: \\
General Framework  & Sparse Inverse   & $l_1$-regularized\\
 & Covariance Estimation &  Logistic Regression\\
\hline
\hline
$\bfx$ - unknown vector.& $A$ - unknown matrix. & $\bfw$ - unknown vector.\\\hline
$n$ - dimension of $\bfx$ & $n$ - dimension of $A$ & $n$ - dimension of $\bfw$ \\
($\bfx\in\mathbb{R}^n$).& ($A\in\mathbb{R}^{n\times n}$). & ($\bfw\in\mathbb{R}^n$).\\
\hline
 & Data samples and matrix & Data samples and matrix\\
 & $\left\{\ves{y}{i}\right\}_{i=1}^{m} \in \real{n}$, $S\in \real{n\times n}$  & $\left\{\ves{x}{i}\right\}_{i=1}^{m} \in \real{n}$, $X\in\real{n\times m}$,$\ve{y}\in\real{m}$\\
\hline
\end{tabular}
\caption{The notation used in the different sections of this paper.}
\end{table*}

\section{The sparse inverse covariance estimation problem}

Estimating the parameters of a multivariate Gaussian (Normal) distribution is a common problem in many applications in machine learning, computational biology, and other fields \cite{banerjee2008model}. Given a set of samples $\{\ve{y}_i\}_{i=1}^m\in\mathbb{R}^n$, where \gaussian{\ve{y}_i}{\mu}{\Sigma}, the objective is to estimate the mean $\mathbb{\mu}\in\mathbb{R}^{n}$, and either the covariance matrix $\Sigma\in\mathbb{R}^{n\times n}$ or its inverse $\Sigma^{-1}$, which is also called the precision matrix.
In particular, the \emph{inverse} of the covariance matrix, which represents the underlying graph of a Gaussian Markov random field (GMRF), is useful in many applications \cite{GMRFBOOK}.

Both the mean $\mu$ and the covariance $\Sigma$ are often estimated using the maximum likelihood estimator (MLE), given the samples $\{\ve{y}_i\}_{i=1}^m$. The MLE
aims to maximize the probability of sampling $\{\ve{y}_i\}_{i=1}^m$ given the parameters. In the Gaussian case, this leads to the maximization of the density function of the Normal distribution
\begin{equation}\label{eq:MLEorig}
\begin{array}{rl}
\hat\mu,\hat\Sigma^{\mbox{\tiny{MLE}}} =& \displaystyle{\argmax_{\Sigma,\mu} \prod_{i=1}^{m}{\mathbb{P}(\bfy_i|\Sigma,\mu)}}\\
= & \displaystyle{\argmax_{\Sigma,\mu} \prod_{i=1}^{m}\frac{1}{\sqrt{(2\pi)^m\det{(\Sigma)}}}\exp\left(-\frac{1}{2} (\bfy_i-\mu)^T\Sigma^{-1}(\bfy_i-\mu)\right)}.
\end{array}
\end{equation}
This yields $\hat{\mu} = \frac{1}{m}\sum_{i=0}^m\ve{y}_i$ as estimation for the mean and\footnote{Equation \eqref{eq:cov_estimate} is the standard MLE estimator, derived from \eqref{eq:MLEorig}. However, sometimes the unbiased MLE estimation is preferred, where $m-1$ replaces $m$ in the denominator.}
\begin{equation}\label{eq:cov_estimate}
\mat{S}\mydef \hat{\mat{\Sigma}}^{\mbox{\tiny{MLE}}} = \frac{1}{m}\sum_{i=0}^m(\ve{y}_i-\hat{\mu})(\ve{y}_i-\hat{\mu})^T,
\end{equation}
which is also called the empirical covariance matrix. More specifically, by applying $-\log$ to the MLE objective in \eqref{eq:MLEorig} and minimizing it over the inverse covariance matrix we get that $\mat{\Sigma}^{-1}$ is estimated by solving the optimization problem
\begin{equation}\label{eq:MLE}
\min_{\mat{A}\succ 0} f(\mat{A}) \mydef \min_{\mat{A}\succ 0} -\log(\det{\mat{A}})+\tr({\mat{S}\mat{A}}),
\end{equation}
which also leads to \eqref{eq:cov_estimate}.

However, if the number of samples is smaller than the problem dimension, i.e., $m<n$, then \mat{S} in \eqref{eq:cov_estimate} is rank deficient and not invertible, whereas the true $\mat{\Sigma}$ is assumed to be full-rank and positive definite. Nevertheless, in this case one can reasonably estimate $\mat{\Sigma}^{-1}$ by adding further assumptions.
It can be observed in the probability density function in \eqref{eq:MLEorig} that if $(\mat{\Sigma}^{-1})_{ij} = 0$, then the random variables in the $i$-th and $j$-th entries of a vector \gaussian{\ve{y}}{\mu}{\mat{\Sigma}} are conditionally independent, given that the rest of the variables are known \cite{dempster1972covariance}. Therefore, one may look at $\Sigma^{-1}$ as a direct dependency matrix where each of its off-diagonal non-zeros indicates a direct dependency between two variables. For this reason, many applications adopt the notion of estimating a \emph{sparse} inverse of the covariance, $\mat{\Sigma}^{-1}$. (Note that in most cases $\mat{\Sigma}$ remains dense.)
For this purpose, we follow \cite{banerjee2008model,banerjee2006convex,d2008first}, and minimize \eqref{eq:MLE} with a sparsity-promoting $l_1$ prior:
\begin{equation}
\label{eq:MLE_l1_prior}
\min_{\mat{A}\succ 0}\;  F(\mat{A}) \mydef  \min_{\mat{A}\succ 0}\;  f(\mat{A}) + \lambda\|\mat{A}\|_1.
\end{equation}
Here, $f(\mat{A})$ is the MLE objective defined in \eqref{eq:MLE}, $\|\mat{A}\|_1\equiv\sum_{i,j}{|a_{ij}|}$, and $\lambda>0$ is a regularization parameter. The sparsity assumption is justified when most variables are directly statistically dependent on only a small number of variables, and thus conditionally independent of the rest. Problem \eqref{eq:MLE_l1_prior} is also called \emph{Covariance Selection} \cite{dempster1972covariance} and it has a unique solution \cite{banerjee2008model,d2008first}. It is an instance of \eqref{eq:FX}, so it is non-smooth and convex, but unlike \eqref{eq:FX} it is also constrained to the positive definite domain.

Many methods were recently developed for solving \eqref{eq:MLE_l1_prior}---see \cite{banerjee2006convex,banerjee2008model,d2008first,GLASSO,guillot2012iterative,DCQUIC,BIGQUIC,QUIC,mazumder2012exact,ORTHANT,BCDIC} and references therein.
However, as mentioned earlier, in this work we are interested in efficiently solving large scale instances of \eqref{eq:MLE_l1_prior}, where $n$
is large such that $O(n^2)$ variables cannot fit in memory (we assume that the data samples $\{\ve{y}_i\}_{i=1}^m$ do fit in memory).
This makes the solution of \eqref{eq:MLE_l1_prior} particularly challenging, since the gradient of $f(\mat{A})$ includes $A^{-1}$, which is a dense $n\times n$ matrix, coming from the $\log\det$ term. Because of this, most of the existing methods cannot be used to solve \eqref{eq:MLE_l1_prior}, as they use the full gradient of $f(\mat{A})$. The same applies for the strategies of \cite{banerjee2008model,GLASSO} that target the \emph{dense} covariance matrix itself rather than its inverse, using the dual formulation of \eqref{eq:MLE_l1_prior}.
Two exceptions are (1) BigQUIC - a proximal Newton approach in \cite{BIGQUIC}, which was made suitable for large-scale matrices by treating the Newton direction problem in blocks, and (2) a Block-Coordinate-Descent for Inverse Covariance Estimation (BCD-IC) method \cite{BCDIC} that directly treats \eqref{eq:MLE_l1_prior} in blocks. In the following sections we briefly describe the proximal Newton approach for \eqref{eq:MLE_l1_prior}, and review the BCD-IC method of \cite{BCDIC}. Following that, we describe how to accelerate BCD-IC by our multilevel framework, and show improvements for it in the case where problem \eqref{eq:MLE_l1_prior} is solved for a given support---similarly to problem \eqref{eq:FX_coarse} which is constrained to a given support.

\subsection{Proximal Newton methods for sparse inverse covariance estimation}

A few of the current state-of-the-art methods \cite{DCQUIC,BIGQUIC,QUIC,ORTHANT} for \eqref{eq:MLE_l1_prior} involve the ``proximal Newton'' approach described earlier in Section \ref{sec:relaxations_general}. To obtain the Newton descent direction, the smooth part $f(\mat{A})$ in \eqref{eq:MLE_l1_prior} is replaced by a second order Taylor expansion, while the non-smooth $l_1$ term remains intact. This requires computing the gradient and Hessian of $f(\mat{A})$, which are given by
\begin{equation}\label{eq:gradHes}
\nabla f(\mat{A}) = \mat{S}-\mat{A}^{-1}, \quad\quad \nabla^2 f(\mat{A}) = \mat{A}^{-1}\otimes \mat{A}^{-1},
\end{equation}
where $\otimes$ is the Kronecker product. The presence of $\mat{A}^{-1}$ in the gradient not only imposes memory problems in
large scales, it is also expensive to compute. Therefore, the advantage of the proximal Newton approach here is the low overhead: by calculating the $\mat{A}^{-1}$ in $\nabla f(\mat{A})$, we also get the information needed to apply the Hessian \cite{BIGQUIC,QUIC,ORTHANT}.

Similarly to \eqref{eq:F_tilde_general}, the Newton direction $\mat{\Delta}^{(k)}$ is the solution the LASSO problem,
\begin{equation}\label{eq:Quad_MLE_l1_prior}
\begin{array}{rl}
\mat{\Delta}^{(k)} =& \displaystyle{\argmin_{\mat{\Delta}\in\mathbb{R}^{n\times n}}\;\tilde F(\mat{A}^{(k)}+ \mat{\Delta})}\\
= & \displaystyle{\argmin_{\mat{\Delta}\in\mathbb{R}^{n\times n}}f(\mat{A}^{(k)})+\tr(\mat{\Delta}(\mat{S}-\mat{W})) +\frac{1}{2}\tr(\mat{\Delta}\mat{W}\mat{\Delta}\mat{W}) +  \lambda\|\mat{A}^{(k)}+\mat{\Delta}\|_1},
\end{array}
\end{equation}
where $\mat{W} = \left(\mat{A}^{(k)}\right)^{-1}$. The gradient and Hessian of $f(\mat{A})$ in \eqref{eq:gradHes} are featured in the second and third terms in \eqref{eq:Quad_MLE_l1_prior}, respectively. Once the direction $\mat{\Delta}^{(k)}$ is computed, it is added to $\mat{A}^{(k)}$ employing a linesearch procedure to sufficiently reduce the objective in \eqref{eq:MLE_l1_prior} while ensuring positive definiteness. To this end, the updated iterate is $\mat{A}^{(k+1)} = \mat{A}^{(k)} + \alpha\mat{\Delta}^{(k)}$, where $\alpha>0$ may be obtained using Armijo's rule \cite{QUIC}.

\subsection{Restricting the updates to free sets} \label{sec:FreeSet}

In addition, \cite{QUIC} introduced a crucial step: restricting of the Newton direction in \eqref{eq:Quad_MLE_l1_prior} to
a ``free set'' of variables, while keeping the rest as zeros. The free set of a matrix $\mat{A}$ is defined as
\begin{equation}\label{eq:active_set}
\Active(\mat{A}) = \left\{(i,j): \mat{A}_{ij}\neq0 \vee |(\mat{S} - \mat{A}^{-1})_{ij}| > \lambda \right\}.
\end{equation}
If one solves \eqref{eq:Quad_MLE_l1_prior} with respect only to the variables outside this free set, they all remain zero, suggesting that it is worthwhile to (temporarily) restrict \eqref{eq:Quad_MLE_l1_prior} only to the variables in this set \cite{QUIC}. This reduces the computational complexity of most LASSO solvers: given the matrix $W$, the Hessian term in \eqref{eq:Quad_MLE_l1_prior} can be calculated in $O(Kn)$ operations instead of $O(n^3)$, where $K = |\Active\left(A^{(k)}\right)|$. This saves significant computations in each Newton update, and at the same time does not significantly increase the number of iterations needed for convergence \cite{BCDIC}.

\section{Block coordinate descent for sparse inverse covariance estimation (BCD-IC)}\label{sec:BCDIC}
In this section we review the iterative Block Coordinate Descent method for solving large-scale instances of \eqref{eq:MLE_l1_prior}.
In this method, we iteratively update the solution in blocks of matrix variables, where each block is defined as the free set of variables within a
relatively small subset of columns of $\mat{A}$. We iterate over all blocks, and in turn minimize \eqref{eq:MLE_l1_prior} restricted to each block by using a quadratic approximation, while the other matrix entries remain fixed. Since we consider one sub-problem at a time, we can fully store the gradient and Hessian for each block, assuming that the blocks are chosen to be small enough.

We limit our blocks to subsets of columns because this way, the corresponding portion of the gradient \eqref{eq:gradHes}
can be computed as solutions of linear systems. Because the matrix is symmetric, the corresponding rows are updated simultaneously.
Figure \ref{fig:BCD_Sweep} shows an example of a BCD iteration where the subsets of columns are chosen in sequential order. In practice, theses subsets can be non-contiguous and vary between the BCD iterations. We elaborate later on how to partition the columns, and on some  advantages of this block-partitioning.
Partitioning the matrix into small blocks enables our method to solve \eqref{eq:MLE_l1_prior} in high dimensions (up to millions of variables), requiring $O(n^2/p)$ additional memory, where $p$ is the number of blocks (that is in addition to the memory needed for storing the iterated solution $\mat{A}^{(k)}$ itself, and the data $\{\ve{y}_i\}_{i=1}^m$).

\begin{figure}[!t]
\centering
\includegraphics[width=120mm]{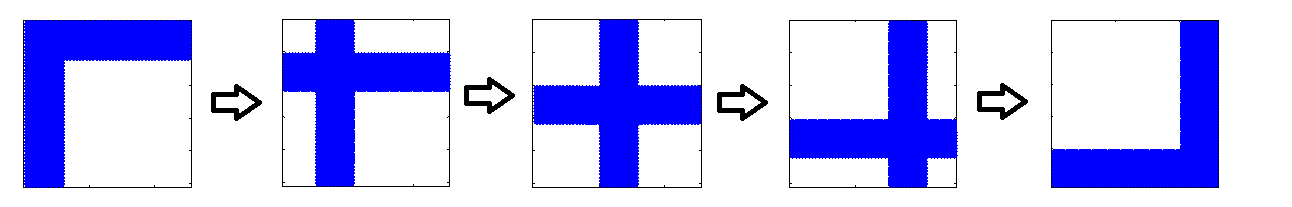}
\caption{Example of a BCD iteration. The blocks are treated successively.}
\label{fig:BCD_Sweep}
\end{figure}

\subsection{BCD-IC iteration}
We now describe a BCD-IC iteration, in which we divide the matrix into blocks, and iteratively update the
solution matrix block by block. Assume that the set of columns $\{1,...,n\}$ is divided into $p$ subsets $\{I_{j}\}_{j=1}^{p}$, where $I_j$ contains the indices
of the columns that comprise the $j$-th block.  We denote the updated matrix after treating the $j$-th block at iteration $k$ by $\mat{A}^{(k)}_j$ and the next iterate is defined once we finish treating all blocks, i.e., $\mat{A}^{(k+1)}=\mat{A}_p^{(k)}$. However, for simplicity of notation, let us denote the updated matrix $\mat{A}_{j-1}^{(k)}$, before treating block $j$ at iteration $k$, by $\matt{A}$.

To update block $j$, we form and minimize a quadratic approximation of problem \eqref{eq:MLE_l1_prior}, restricted to the rows/columns in $I_j$:
\begin{equation}
\label{eq:sub_Quad_MLE_l1_prior}
\min_{\mat{\Delta}_j}\;\tilde F(\matt{A}+\mat{\Delta}_j),
\end{equation}
where $\tilde F(\cdot)$ is the quadratic approximation of \eqref{eq:MLE_l1_prior} around $\matt{A}$, similarly to \eqref{eq:Quad_MLE_l1_prior}, and $\mat{\Delta}_j$ has non-zero entries only in the rows/columns in $I_j$. In addition, we restrict the non-zeros of $\mat{\Delta}_j$ to the free set defined in \eqref{eq:active_set}. That is, $\mat{\Delta}_j$ in \eqref{eq:sub_Quad_MLE_l1_prior} is restricted to the free set
\begin{equation}\label{eq:active_set_j}
\Active_{I_j}(\matt{A}) = \Active(\matt{A})\cap \left\{(i,k): i\in I_j \vee k\in I_j\right\},
\end{equation}
while all other entries in $\mat{\Delta}_j$ are fixed to zero. To calculate \eqref{eq:active_set_j}, we check the condition in \eqref{eq:active_set} only in the columns $I_j$, which requires the gradient \eqref{eq:gradHes} for block $I_j$.
For that, we calculate the columns $I_j$ of $\matt{A}^{-1}$ by solving $|I_j|$ linear systems, with the canonical vectors $\ve{e}_l$ as right-hand-sides for each $l\in I_j$, i.e., $(\matt{A}^{-1})_{I_j}=\matt{A}^{-1}\ve{E}_{I_j}$.
The solution of these linear systems is one of the main computational tasks of our algorithm, and can be achieved in various ways.
For large dimensions, iterative methods such as Conjugate Gradients (CG) are usually preferred, possibly with preconditioning \cite{Saad03}.

\subsubsection{Treating a block-subproblem by Newton's method} \label{sec:Newton}

To get the Newton direction for the $j$-th block, we solve the LASSO problem \eqref{eq:sub_Quad_MLE_l1_prior} by using PCD accelerated by non-linear Conjugate Gradients (PCD-CG) \cite{ELADZIBULEVSKY,TY}. For a complete and detailed description of this algorithm see the Appendix of \cite{BCDIC}.

Let us denote $\mat{W}=\matt{A}^{-1}$. To apply a PCD-CG iteration,
we need to calculate the objective of \eqref{eq:sub_Quad_MLE_l1_prior} and its gradient efficiently. For that, we need to calculate the matrices $\mat{W}\mat{\Delta}_j\mat{W}$, $\mat{S}-\mat{W}$, and $l_1$ term only at the entries \eqref{eq:active_set_j}, where $\mat{\Delta}_j$ is non-zero. We compute only the columns of $\mat{\Delta}_j$, because the rows are obtained by symmetry. The main computational task here involves the ``Hessian-vector product'' $\mat{W}\mat{\Delta}_j\mat{W}$. For that, we reuse the $I_j$ columns of $\matt{A}^{-1}$ calculated for obtaining \eqref{eq:active_set_j}, denoted now by $\mat{W}_{I_j}$. Since we only need the result in the columns $I_j$, we observe that $\left(\mat{W}\mat{\Delta}_j\mat{W}\right)_{I_j} =  \mat{W}\mat{\Delta}_j\mat{W}_{I_j}$, and the product $\mat{\Delta}_j\mat{W}_{I_j}$ can be computed efficiently because $\mat{\Delta}_j$ is sparse.

In order to compute $\mat{W}(\mat{\Delta}_j\mat{W}_{I_j})$ for the entries in \eqref{eq:active_set_j}, we follow the idea of \cite{BIGQUIC} and use the rows (or columns) of $\mat{W}$ that are represented in \eqref{eq:active_set_j}. Besides the columns $I_j$ of $\mat{W}$ we also need the ``neighborhood'' of $I_j$ defined as
\begin{equation}\label{eq:extended_set}
N_j = \left\{ i: \exists k \notin I_j : (i,k)\in \Active_{I_j}(\mat{A}) \right\}.
\end{equation}
The size of this set will determine the amount of additional columns of $\mat{W}$ that we need, and therefore we wish it to be as small as possible.
To achieve that, we follow \cite{BIGQUIC} and define the blocks $\{I_j\}$ using clustering methods, which aim to partition the columns/rows into disjoint subsets, such that there are as few non-zero entries as possible outside the diagonal blocks of the matrix that correspond to each subset. In our notation, we aim that the size of $N_j$ will be as small as possible for every block $I_j$, which is chosen to be relatively small. We use METIS \cite{METIS}, but other methods may be used instead.
Note that only $|N_j|\times|N_j|$ numbers out of $\mat{W}_{N_j}$ are necessary for computing the relevant entries of $\mat{W}(\mat{\Delta}_j\mat{W}_{I_j})$. However, there might be situations where the matrix has a few dense columns, resulting in some sets $N_j$ of size $O(n)$. Computing $\mat{W}_{N_j}$ for those sets is not possible because of memory limitations. This case is treated separately---see \cite{BCDIC} for details.%Section \ref{sec:dense_cols} in the Appendix for details.
\subsubsection{Updating the solution with line-search}\label{sec:linesearch}

Denote the solution of the Newton direction problem \eqref{eq:sub_Quad_MLE_l1_prior} by $\mat{\Delta}_j^{(k)}$. Now we wish to update the matrix $\mat{A}^{(k)}_j = \mat{A}^{(k)}_{j-1} + \alpha\mat{\Delta}_j^{(k)}$, where $\alpha>0$ is obtained by a linesearch procedure, which requires evaluating the objective of \eqref{eq:MLE_l1_prior} for several values of $\alpha$.

First, for any sparse matrix $A$ the cost of computing the trace and $l_1$ terms in \eqref{eq:MLE_l1_prior} is proportional to the number of non-zero entries in $A$ (and the number of sample vectors). However, calculating the determinant of a general $n\times n$ sparse matrix for evaluating the $\log\det$ term of \eqref{eq:MLE_l1_prior} may be costly. This may be done by using a sparse Cholesky factorization, but here we assume that $n$ is too large for that.
%In \cite{BIGQUIC}, where a general Newton direction $\mat{\Delta}^{(k)}$ as in \eqref{eq:Quad_MLE_l1_prior} is considered, this is done by solving $n-1$ linear systems of decreasing sizes from $n-1$ to 1 for each sample of $\alpha$.
In our case, however, since $\mat{\Delta}^{(k)}_j$ has a special block structure, we can reduce the $\log\det$ term to a log-determinant of a small dense $|I_j| \times|I_j|$ matrix, and compute it efficiently.

Let us introduce a partitioning of any matrix \mat{A} into blocks, according to a subset of indices $I_j\subseteq\{1,...,n\}$.
Assume without loss of generality that the matrix \mat{A} has been permuted such that the columns/rows with indices in $I_j$ appear first, and let
\begin{equation}\label{eq:partitioning}
\mat{A} =
\left[
\begin{array}{c|ccc}
\mat{A}_{11} &   & \mat{A}_{12} &  \\ \hline
             &   &              &  \\
\mat{A}_{21} &   & \mat{A}_{22} &  \\
             &   &              &
\end{array}
\right]
\end{equation}
be a partitioning of $\mat{A}$. The sub-matrix $\mat{A}_{11}$ corresponds to the elements in rows and columns $I_j$ in \mat{A}. According to the Schur complement \cite{Saad03}, for any invertible matrix and block-partitioning as above, the following holds:
\begin{equation}\label{eq:Schurs_lemma}
\log\det(\mat{A}) = \log\det(\mat{A}_{22}) + \log\det(\mat{A}_{11} - \mat{A}_{12}\mat{A}_{22}^{-1}\mat{A}_{21}).
\end{equation}
Furthermore, for any symmetric matrix $\mat{A}$ the following applies:
\begin{equation}\label{eq:Schurs_PDness}
\mat{A}\succ 0 \Leftrightarrow  \mat{A}_{22} \succ 0 \mbox{ and } \mat{A}_{11} - \mat{A}_{12}\mat{A}_{22}^{-1}\mat{A}_{21}\succ 0.
\end{equation}
Using the above notation and partitioning for $\matt{A}$ and $\mat{\Delta}_j^{(k)}$, we write using \eqref{eq:Schurs_lemma}:
\begin{equation}
\log\det{(\matt{A}+\alpha\mat{\Delta}_j^{(k)})}= \log\det{(\matt{A}_{22})}+ \log\det(\mat{B}_0 +\alpha\mat{B}_1 + \alpha^2 \mat{B}_2)
\label{eq:logdet_blocks}
\end{equation}
where $\;\;\mat{B}_0 =  \matt{A}_{11} - \matt{A}_{12}\matt{A}_{22}^{-1}\matt{A}_{21}$, $\;\;\mat{B}_1 =  \mat{\Delta}_{11} -\mat{\Delta}_{12}\matt{A}_{22}^{-1}\matt{A}_{21} - \matt{A}_{12}\matt{A}_{22}^{-1}\mat{\Delta}_{21}$, and \newline $\mat{B}_2 =  -\mat{\Delta}_{12}\matt{A}_{22}^{-1}\mat{\Delta}_{21}$. (Note that here we replaced $\mat{\Delta}_j^{(k)}$ by $\mat{\Delta}$ to simplify notation.) If the set $I_j$ is relatively small, then so are the matrices $\mat{B}_i\in\mathbb{R}^{|I_j|\times|I_j|}$ in \eqref{eq:logdet_blocks}, and given these matrices we can easily compute the objective $F(\cdot)$. Furthermore, following \eqref{eq:Schurs_PDness}, the constraint $\matt{A} + \alpha\mat{\Delta}_j^{(k)}\succ 0$ involved in a linesearch for $\mat{\Delta}_j^{(k)}$ is equivalent to $\mat{B}_0 +\alpha \mat{B}_1 + \alpha^2 \mat{B}_2 \succ 0$, assuming that $\matt{A}_{22}\succ 0$. Calculating the matrices $\mat{B}_i$ in \eqref{eq:logdet_blocks} seems expensive, but they can be efficiently obtained from the previously computed matrices $W_{I_j}$ and $W_{N_j}$ mentioned earlier---see Appendix \ref{sec:computingBi} for details. Therefore, computing \eqref{eq:logdet_blocks} can be achieved in $O(|I_j|^3)$ time complexity.

Using the properties described above, we can easily apply a linesearch for $\mat{\Delta}_j^{(k)}$, and guarantee in every update that our iterated solution matrix $\matt{A}$ remains positive definite throughout the iterations. More specifically,  in this work we use a variant of the Armijo backtracking rule that was also suggested in \cite{yun2011coordinate}, approximately minimizing the objective over $\alpha$. That is, we choose $\alpha_0=1$, and $0<\beta<1 (=0.5)$, and examine the values of $F(\cdot)$ for $\alpha = \alpha_0\beta^i$ for $i=0,1,2,...$. We iterate over $i$ to find a point where $F(\cdot)$ is minimized over the samples $\alpha_0\beta^i$ subject to $\matt{A}+\alpha\mat{\Delta}_j^{(k)} \succ 0$.
%Since $F(\cdot)$ is convex, and the positive definite subset is convex, we terminate this procedure once we find a positive definite point $\matt{A}$ for a given $i$ where the objective is increased for $i+1$.
This requires the initialization of the algorithm, $\mat{A}^{(0)}$, to be positive definite.

\begin{algorithm}\label{alg:BCDiteration}
\small
\Indp
\DontPrintSemicolon
\caption{\emph{Block Coordinate Descent for Inverse Covariance Estimation}}
%\KwIn{Samples:$ \{\ve{x}_i\}_{i=1}^m$, Regularization parameter $\lambda$}
%\KwOut{An estimate for the sparse inverse of the covariance: $\mat{A}$}
\KwSty{Algorithm: BCD-IC($\mat{A}^{(0)}$,$\{\ve{x}_i\}_{i=1}^m$,$\lambda$)}\;
\For{$k=0,1,2,...$}{
Calculate clusters of elements $\{I_j\}_{j=1}^p$ based on $\mat{A}^{(k)}$.\;
\textit{\% Denote: $\mat{A}^{(k)}_0 = \mat{A}^{(k)}$}\;
\For{$j=1,...,p$}{
Compute $W_{I_j} = \left((\mat{A}^{(k)}_{j-1})^{-1}\right)_{I_j}$. \textit{\% solve $|I_j|$ linear systems}\;
Define $\Active_{I_j}\left(\mat{A}^{(k)}_{j-1}\right)$ as in \eqref{eq:active_set_j}, and define the set $N_j$ in \eqref{eq:extended_set}.\;
Compute $W_{N_j} = \left((\mat{A}^{(k)}_{j-1})^{-1}\right)_{N_j}$. \textit{\% solve $|N_j|$ linear systems}\;
Find the Newton direction $\mat{\Delta}_j^{(k)}$ by solving \eqref{eq:sub_Quad_MLE_l1_prior}.\;
Update the solution: $\mat{A}^{(k)}_j = \mat{A}^{(k)}_{j-1} + \alpha\mat{\Delta}_j^{(k)}$ by linesearch.
}
\textit{\% Denote: $\mat{A}^{(k+1)} = \mat{A}^{(k)}_p$}\;
}
\end{algorithm}

\subsection{Convergence of BCD-IC}\label{sec:convergence-analysis}

The paper \cite{BCDIC} states the following theorem:
\begin{theorem}
\label{thm:Convergence}
In Algorithm \ref{alg:BCDiteration}, the sequence $\left\{\mat{A}_j^{(k)}\right\}$ converges to the global optimum of \eqref{eq:MLE_l1_prior}.
\end{theorem}

The proof of this theorem is based on the analysis of \cite{BCDTsengYun,QUIC}.
In \cite{BCDTsengYun}, a general block-coordinate-descent approach is analyzed to solve minimization problems of the form $F(\mat{A})=f(\mat{A})+\lambda h(\mat{A})$, where $f(\cdot)$ is a smooth function and $h(\cdot)$ is a separable convex function, which in our case are \eqref{eq:MLE} and $\|\mat{A}\|_1$, respectively. Although this setup fits the the problem \eqref{eq:MLE_l1_prior}, \cite{BCDTsengYun} treats the problem in the $\real{n}$ domain, while the minimization in \eqref{eq:MLE_l1_prior} is being constrained over the symmetric positive definite domain. To overcome this limitation, the authors in \cite{QUIC} extended the analysis in \cite{BCDTsengYun} to treat the specific constrained problem \eqref{eq:MLE_l1_prior}. Except for the solution of the inner LASSO problems, \cite{QUIC} is equivalent to BCD-IC using only one block which contains all variables. Hence, the convergence proof of BCD-IC in \cite{BCDIC} extends that of \cite{QUIC}.

\section{Application of the multilevel framework to sparse inverse covariance estimation}\label{sec:MLInvCov}

Given an iterate $A^{(k)}$, proximal Newton methods like \cite{QUIC,BIGQUIC,ORTHANT} or BCD-IC limit their Newton directions to the variables in the free set, saving a significant amount of computations.
However, in \cite{QUIC} it is shown that if $A^{(k)}$ is far from the optimal solution $A^{*}$, then $|\Active(A^{(k)})|$ may be several times larger than $|\Active(A^{*})|$, since the entries of the gradient of $A^{(k)}$ are typically large.
This may impose extensive computations: we get larger and more difficult Newton problems.
Moreover, for large scales, the cost of solving linear systems for the gradients is directly proportional to the number of non-zeros in the matrices. As the iterates progress, the support size of the iterated matrices reduces, until it converges to that of $A^{*}$. As mentioned in the first part of this paper, if we knew the non-zeros of $A^{*}$, solving \eqref{eq:MLE_l1_prior} would require less computations---this again motivates the use of our multilevel framework for this problem.

Even though \eqref{eq:MLE_l1_prior} has many unique properties, we apply our multilevel framework in Algorithm \ref{alg:ML} as is, using a method like BCD-IC or \cite{QUIC,BIGQUIC,ORTHANT} as the relaxation. In particular, in this paper we use BCD-IC, because we target large-scale problems. We now elaborate on how to define the ingredients of the multilevel approach, the restricted relaxation, selection of hierarchy $\{\C_l\}_{l=0}^L$, and parameters $\nu$ and $\nu_c$.

Similarly to \eqref{eq:FX_coarse}, we define a coarse problem at level $l$ by limiting problem \eqref{eq:MLE_l1_prior} to a subset of entries, denoted by $\C_l$
\begin{equation}
\label{eq:MLE_l1_prior_C}
 \min_{A\succ 0,\;\support(A)\in\C_l}\;  f(A) + \lambda\|A\|_1,
\end{equation}
where $f(A)$ is defined in \eqref{eq:MLE}. To solve problem \eqref{eq:MLE_l1_prior_C} using the proximal Newton methods \cite{QUIC,BIGQUIC}, for example, one may restrict the Newton direction to $\Active(A)\cap\C_l$ instead of $\Active(A)$. If one uses BCD-IC, the same applies for the free set \eqref{eq:active_set_j}. This allows a significant improvement for BCD-IC in this case: unlike the original case where we need all the rows of $W_{I_j}$ for checking \eqref{eq:active_set_j}, now we \emph{apriori} need only the rows of $W_{I_j}$ that are represented in $\C_l$, because \eqref{eq:active_set_j} is restricted to $\C_l$. Using Schur complement properties, we get those rows of $W_{I_j}$ using only the columns of $A^{-1}$ that correspond to the neighborhood of $I_j$ in $\C_l$, and as a result, we no longer need the solution of $|I_j|$ linear systems. A detailed description is given in Appendix \ref{sec:reducingGradient}. This significantly reduces the cost of a BCD-IC relaxation.

Next, within each of the relaxation methods mentioned, there is an inherited selection of variables in the form of free set. This affects the choice for the multilevel hierarchy.  As in \eqref{eq:hierarchy1}, we choose all available variables for $\C_0$, and apply the finest level relaxation for \eqref{eq:MLE_l1_prior} without an additional constraint. However, if for example we let $\C_1$ include half of the variables, the free set in the corresponding relaxation for $\C_1$ will most likely be nearly the same as for $\C_0$, as the size of a free set is typically much smaller than half of all the variables. Therefore, in order to inforce a significant reduction in the problem size, we need to select a subset of the free set. So, to define the multilevel hierarchy \eqref{eq:hierarchy1}, we use the free set calculated in the fine level relaxation in a previous cycle\footnote{Calculating a free set is a relatively expensive procedure. Therefore, we use the most relevant free set and gradient that we have from previous computations. If we use \cite{BIGQUIC} as relaxation, that would be the free set of the fine level relaxation from the previous cycle. If we use BCD-IC, then we use the union of free sets determined for all the blocks as a free set. As an initial free set, one may use $\{(i,j): |S_{ij}| > \lambda\}\cup\support{(A^{(0)})}$.}. Now, let $\Active_k$ be this free set, then similarly to Section \eqref{sec:likely}, for $\C_1$ we first choose the entries in $\support(\mat{A}^{(k)})$, and then choose $\lceil\frac{1}{2}|\Active_k|\rceil - |\support(\mat{A}^{(k)})|$ additional variables with the highest absolute value of the gradient. For the rest of the levels, we chose $|\C_{l+1}| = \max\{\lceil\frac{1}{2}|\C_l|\rceil,|\support{(\mat{A}^{(k)})}|\}$ for $1\le l<L$, based on the support and the size of the gradient, until $\C_L = \support{(\mat{A}^{(k)})}$.

Lastly, we apply our multilevel framework for solving \eqref{eq:MLE_l1_prior} using $\nu=\nu_c=1$ relaxations in Algorithm \ref{alg:ML}. We apply only one relaxation on each level because the methods mentioned are quite effective and require only a few iterations to solve \eqref{eq:MLE_l1_prior}. That is not only because these methods use second order information, but also since there is no point in solving \eqref{eq:MLE_l1_prior} up to high accuracy because of statistical noise. The problem with these methods is that each of their iterations is expensive, especially if the free set is large. Therefore, to get the most out of our multilevel structure, and especially not to overdo the coarsest level solution, we apply only one relaxation on each level, including the coarsest.

\subsection{Convergence of the multilevel framework for inverse covariance estimation}\label{sec:convergenceInvCov}
As noted before, problem \eqref{eq:MLE_l1_prior} is different from \eqref{eq:FX} because it has the positive definiteness constraint. Still, by Corollary \ref{prop:Convergence3}, Algorithm \ref{alg:ML} converges when it is applied with any relaxation method that satisfies the conditions of Theorem \ref{prop:Convergence2}, and Lemmas \ref{lem:monotonicity} and \ref{lem:sufficient_reduction}. Indeed, the methods QUIC \cite{QUIC} (and BIG-QUIC \cite{BIGQUIC}) and BCD-IC \cite{BCDIC} satisfy those conditions. By Lemma 2 in \cite{QUIC} we know that all iterates of QUIC are contained in the compact level set $U=\{A:f(A)<f(A_0),A\succ0\}$, so throughout the iterations $\theta_{min}I\preceq A\preceq\theta_{max}I$, and therefore the Hessian $\nabla^2 f$, which is also used as the iteration matrix $H^{(k)}$, is bounded from below and from above. Also, Propositions 3-5 in \cite{QUIC} include the result of Lemma \ref{lem:monotonicity}, and those appear in \cite{BCDIC}. Lemma \ref{lem:sufficient_reduction} holds for both \cite{QUIC} and \cite{BCDIC} because both satisfy its conditions and apply \eqref{eq:F_tilde_general}-\eqref{eq:Line-search} in blocks and cover all variables periodically (by satisfying the Gauss-Seidel rule in \cite{BCDTsengYun}). Finally, the conditions of Theorem \ref{prop:Convergence3} hold for \eqref{eq:MLE_l1_prior} and the relaxations \cite{QUIC,BIGQUIC,BCDIC}, and therefore, Algorithm \ref{alg:ML} converges for the solution of \eqref{eq:MLE_l1_prior} when used with either one of the methods \cite{QUIC,BIGQUIC,BCDIC} as a relaxation.

\section{Numerical results: sparse inverse covariance estimation}

In this section we compare the performance of several approaches to our multilevel framework, for solving large-scale instances of \eqref{eq:MLE_l1_prior}.
Our multilevel framework is applied on BCD-IC (and denoted ML-BCD), and compared with ``stand-alone'' BCD-IC (Algorithm \ref{alg:BCDiteration}) and BIG-QUIC \cite{BIGQUIC} (for our tests we adapted the authors' software, which is written and parallelized in C). Furthermore, we include other acceleration frameworks applied on BCD-IC: (1) a continuation strategy (BCDcont.) and a (2) ``divide and conquer'' strategy (DC-BCD) \cite{DCQUIC}. Our MATLAB-based code (including routines in C) is available at: {\url{http://www.cs.technion.ac.il/~eran/}}.

The continuation strategy was generally described at the end of Section \ref{sec:relaxations_general}. More precisely, here we use four decreasing values of regularization parameter $\lambda$ in \eqref{eq:FX}: $\lambda_4>\lambda_3>\lambda_2>\lambda_1$, and apply one BCD-IC iteration for each of those in decreasing order. The final value, $\lambda_1$ is the value in which we want to solve \eqref{eq:FX}. Once this sequence of iterations is over, we keep applying BCD-IC using $\lambda_1$ until convergence. Here we choose $\lambda_4 = \frac{1+\lambda}{2}$, and the rest of the values are linearly spaced between $\lambda_4$ and the original value $\lambda_1$. We note that since we use normalized data, we know that any reasonable $\lambda$ for \eqref{eq:MLE_l1_prior} must be smaller than 1.

The divide and conquer (DC) strategy relies on the fact that if we assume that the solution of \eqref{eq:MLE_l1_prior} is a block-diagonal matrix, then the problem can be separated to a sum of smaller problems according to those blocks. Using this idea, the strategy is applied as follows. We initialize $A^{(0)}$ with a diagonal matrix, and use the graph induced by the free set \eqref{eq:active_set} to create a hierarchy of nested partitioning of the original matrix. We first partition the matrix indices $\{1,...,n\}$ into two sub-sets, then each of those is divided into two to create four sub-sets, and so on until the subsets are small enough (in our tests we terminate the partitioning at size smaller than 4000). Then, assuming that the result of \eqref{eq:MLE_l1_prior} is a block-diagonal matrix with blocks according to the partitioning, we perform an iteration of BCD-IC for each of those block sub-problems separately. Then, we perform a union of blocks opposite to the partitioning process---each block that was split to two sub-blocks is now merged into one again. We apply a BCD-IC iteration for the merged blocks, starting from the block-diagonal approximation obtained from the previous partitioning. We repeat this process of uniting the sub-problems and applying BCD-IC iterations until all the sub-blocks are united to $\{1,...,n\}$ back again. From that point, we apply BCD-IC until convergence. For more details, see \cite{DCQUIC}.

We initialize all methods with the identity matrix. As a stopping criterion for all methods, we follow \cite{BIGQUIC,BCDIC} and use the condition:
$\min_{\bfz}\|\partial F(A^{(k)})\|_1  < \mbox{5e-3}\|A^{(k)}\|_1$, where $\partial F$ is the subdifferential \eqref{eq:subgradient1}. All solutions achieved by all algorithms correspond to objective values $F(A^*)$ which are identical up to several significant digits and have an essentially identical support size. All the experiments were run on a machine with 2 Intel Xeon\footnote{Intel and Xeon are trademarks of Intel Corporation in the U.S. and/or other countries.} E-2650 2.0GHz processors with 16 cores, 64GB RAM and Windows 7 OS. For BCD-IC, we approximate $W_{I_j}$ and $W_{N_j}$ by using conjugate gradients, which we stop once the relative residual drops below $10^{-5}$ and $10^{-4}$, respectively. In addition, we approximate the solution of the Newton direction problem (using PCD-CG) up to a relative precision of $10^{-4}$. The block size of BCD-IC for all tests is 256.

\begin{figure}[!t]
\centering
\includegraphics[width=40mm]{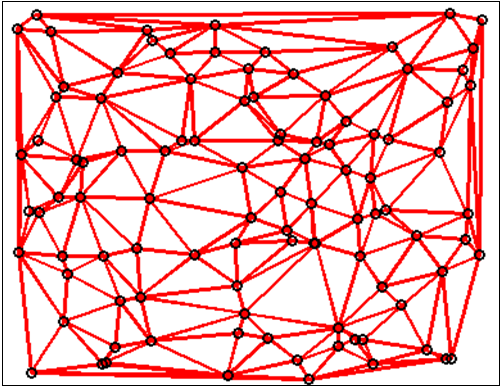}
\caption{Example of a synthetic unstructured planar graph.}
\label{fig:unstructured}
\end{figure}

\subsection{Synthetic experiments} Our first set of experiments is performed on synthetic data, where we use a homogenous random planar graph-Laplacian as a sparse precision matrix. To generate the graph $G(V,E)$ we choose $n$ random points on the unit square as the nodes $V$, and apply a Delaunay triangulation to generate edges in $E$. A 2D example of such graph is shown in figure \ref{fig:unstructured}. Given the graph connectivity, we create a true precision matrix $\Sigma^{-1}$ which is defined by $\forall(i,j)\in E:  (\Sigma^{-1})_{i,j} = -1$, and $(\Sigma^{-1})_{ii} = -\textstyle{\sum_{i\neq j}{(\Sigma^{-1})_{ij}}}$. These matrices are positive semi-definite and have the constant vector as null-space. To make them PD, we remove the points $(i,j)$ whose coordinates are closer than $1/\sqrt n$ to the unit boundary. The resulting matrices have about 6-7 non-zeros per row. To generate the data samples, we randomly generate $m$ Gaussian vectors \gaussian{\ve{v}_i}{0}{I} and form the data samples by $\bfy_i = L^{-1}\bfv_i$, where $L$ is the Cholesky factorization of the true $\Sigma^{-1}$. Following that, the data is normalized to have zero mean and unit variance for each variable (i.e., $diag(S) = I$).
We generate matrices of sizes $n$ varying from approximately 125,000 to 500,000, and generate $m=200$ samples for each. We show the results for four reasonable values of $\lambda$ for each example.

\begin{table}[t]
\centering
\setlength{\tabcolsep}{2pt}
\small \vspace{-10pt}
\label{tb:RPG_results}
\input{SYNTH}
\caption{Large-scale synthetic results for a random planar graph-Laplacian precision matrix.}
\end{table}

In the tables below, we show timings and number of iterations that it took for each method to reach the convergence criterion. We also show the maximal support size that was encountered throughout the iterations (denoted by max-supp), the support size of the minimizer ($\|A^*\|_0$) and an estimate of its condition number ($\kappa(A^*)$), calculated by MATLAB's {\tt condest}. For all methods except DC-BCD, the number of iterations indicate the number of BCD-IC iterations (including coarse levels in ML-BCD and continuation phase in BCDcont.) For DC-BCD, the DC phase is included in the first iteration.

Table \ref{tb:RPG_results} summarizes the results for the synthetic tests. It shows that ML-BCD is the fastest method in almost all the tests. The second-fastest method is DC-BCD, which is also effective at accelerating BCD-IC because the matrices $A^{*}$ in these tests are quite well-conditioned. This agrees with the analysis of DC in \cite{DCQUIC}.  The methods BCDcont and BCD-IC are comparable, with BCD-IC being faster for the high $\lambda$ tests and BCDcont being faster at the harder lower $\lambda$ tests. For the low $\lambda$ BCD-IC and BIG-QUIC see an increase in the maximal observed support compared to the final one. This harms the performance of these methods because some iterations require computations with a less sparse matrix. All three acceleration methods prevent this excess of non-zeros and significantly accelerate BCD-IC as this phenomenon becomes more severe. In particular, ML-BCD managed to speedup BCD-IC approximately 4 times. We note that this is not the only reason for the acceleration---these methods also exploit their ability to reduce the objective value using a relatively small support, which is cheaper to process. BIG-QUIC is significantly outperformed by all methods in these tests.

\subsection{Experiments using real-world data sets} Next, we examine the performance of the methods for large-scale real-world experiments. We use gene expression data sets that are available at the Gene Expression Omnibus (\url{http://www.ncbi.nlm.nih.gov/geo/}), and are reported in \cite{honorio2013inverse,BCDIC}. As before, the data is preprocessed to have zero mean and unit variance for each variable (i.e., $\diag(S) = I$). Table \ref{tb:GEO_results} shows the name-codes of the data sets, as well as the numbers of variables $(n)$ and samples $(m)$, where $m \ll n$. We included three values of $\lambda$ for each data set.

\begin{table}[t]
\centering
\setlength{\tabcolsep}{2pt}
\small \vspace{-10pt}
\label{tb:GEO_results}
\input{GEO-ML}
\caption{Large-scale gene-expression analysis results}
\end{table}

Table \ref{tb:GEO_results} summarizes the results for these real-world experiments. We found these datasets to be more challenging than the synthetic experiments above, which is mostly evident in the higher condition number of the estimated matrices $A^*$. Here, the advantage of the multilevel framework is even more evident. ML-BCD again outperforms the other options by significant factors, and in particular again accelerates BCD-IC by a factor of 3-4.
Because of the relatively high condition numbers, DC-BCD does not perform as well as before, and is generally outperformed by BCDcont. In all cases, the accelerated versions outperformed BCD-IC and BIG-QUIC, because of the decreased support/free set during the solution process. As noted, smaller supports/free sets induce smaller Newton direction problems and faster solution of linear systems, in which each matrix-vector multiplication costs proportionally to the number of non-zeros in the matrix.
By limiting the sparsity of the solution, the acceleration frameworks save significant computations and improve runtime substantially. Out of the three methods, ML-BCD is the most efficient one.

\section{Application of the multilevel framework to $l_1$-regularized logistic regression}
In the next two sections we consider the $l_1$-regularized logistic regression problem. Logistic Regression is a popular classification method in the machine learning literature. Recently, a $l_1$-regularized version of the Logistic Regression problem was introduced to obtain a sparse model, and was shown to be less prone to overfitting \cite{NG2004}.
Given a set of samples $\left\{\ves{x}{i}\right\}_{i=1}^{m} \in \real{n}$ and their respective labels $\left\{y_i\right\}_{i=1}^{m} \in \left\{-1,+1\right\}$, the $l_1$-regularized Logistic Regression classifier is obtained by solving the following optimization problem\footnote{A bias term $b$ can be added to the loss function. Therefore, $\ves{x}{i}^T\ve{w}$ is replaced with $\ves{x}{i}^T\ve{w}+b$ in the loss function and the optimization is held over \ve{w} and $b$.}:
\begin{equation}\label{eq:L1LogReg}
\min_{\ve{w}\in \real{n}} L(\ve{w}) +\normone{\ve{w}} = \min_{\ve{w}\in \real{n}}  C\sum_{i=1}^{m}\log\left(1+e^{-y_i\ves{x}{i}^T\ve{w}}\right) + \normone{\ve{w}},
\end{equation}
where $C$ is a regularization parameter that balances between the sparsity of the model and the loss function $L(\ve{w})$. By dividing this problem by $C$ and setting $\lambda=\frac{1}{C}$, this problem gets the same form as \eqref{eq:FX}.

Many specialized iterative solvers for \eqref{eq:L1LogReg} are available in the literature---see \cite{CDN,GLMNET,NEWGLMNET,IPM,BBR,SCD} and references therein. In this section we demonstrate that our multilevel framework can accelerate the existing solvers for this problem, and in particular, we focus on accelerating the methods \cite{NEWGLMNET} and \cite{CDN}. The method \cite{NEWGLMNET} is a proximal Newton method, as generally described in Section \ref{sec:relaxations_general}, which restricts the Newton direction to a free set (similarly to Section \ref{sec:FreeSet}), and treats the Newton problem using Coordinate Descent. The CDN algorithm \cite{CDN} solves \eqref{eq:L1LogReg} using a Coordinate Descent approach, where in each coordinate update, CDN solves a one-dimensional proximal Newton problem with a one-dimensional line-search procedure. For these purposes, the methods require the gradient and Hessian terms of $L(\ve{w})$:
\begin{equation}
\nabla L(\ve{w}) = C\sum_{i=1}^{m}\left(\tau(y_i\ves{x}{i}^T\ve{w})-1\right)y_i\ves{x}{i} , \qquad
\nabla^2 L(\ve{w}) = CXDX^T,
\label{eq:GradAndHessLogReg}
\end{equation}
where $\tau(s)=\frac{1}{1+e^{-s}}$ is the derivative of the logistic loss function $\log(1+e^{-s})$, $D\in \real{m \times m}$ is a diagonal matrix with elements $D_{ii} =\tau(y_i\ves{x}{i}^T\ve{w})\left(1-\tau(y_i\ves{x}{i}^T\ve{w})\right)$, and $X\in \real{n \times m}$ is a matrix with all the data samples, i.e.,  $X=\left[\ves{x}{1},\dots,\ves{x}{m} \right]$.

To accelerate the convergence of the iterative methods above, we apply Algorithm \ref{alg:ML} using a hierarchy of supports as in \eqref{eq:hierarchy1}. Since the methods \cite{NEWGLMNET} and \cite{CDN} have a restriction to a free set, we apply the same strategy for choosing $\C_1$ as in Section \ref{sec:MLInvCov}.

\section{Numerical results: $l_1$-regularized logistic regression}
We compare between the performances of CDN \cite{CDN} and the newGLMNET \cite{NEWGLMNET} and their accelerated versions denoted by an `ML-' prefix. The stopping criterion of all methods, suggested by \cite{NEWGLMNET}, is $\normone{\nabla^S L(\ve{w}^{(k)})} \le \varepsilon \frac{\min(\#pos,\#neg)}{m}\normone{\nabla^S L(\ve{w}^{(1)})}$, where $\#pos$ and $\#neg$ are the number of positive and negative labels in the samples, and $\nabla^S L(\ve{w})$  is the minimum norm subgradient.
All the methods were implemented in C++ based on the implementation in LIBLINEAR \cite{LIBLINEAR}. All the experiments were executed on a machine with 2 Intel Xeon\footnote{Intel and Xeon are trademarks of Intel Corporation in the U.S. and/or other countries.} E5-2699V3 2.30GHz processors with 36 cores, 128GB RAM, and Linux Cent-OS.

We use the data sets {\tt news20}, {\tt gisette}, {\tt webspam}, {\tt rcv1}, and {\tt epsilon} with values for the regularizer $C$ as reported in \cite{NEWGLMNET}. The $\varepsilon$ value for {\tt news20} dataset is 1e-4, and for the other data sets it is 1e-3. For  ML-newGLMNET,  the number of coordinate descent iterations for the finest and mid-levels are selected small (one and two respectively), and on the coarsest level we allow up to five iterations.

Table \ref{tb:Results} summarizes the results of our experiments, where we present the timing results in seconds and the number of iterations in parentheses. The number of iterations for newGLMNET is for each proximal Newton update, for CDN accounts for updating all the variables ($n$ one-dimensional proximal Newton problems), and for the accelerated methods accounts for the number of ML-cycles.
To save space in Table \ref{tb:Results}, (ML-)newGLMNET is denoted  by (ML-)nGLM.

The multilevel approach shows the best performance and runtime improvement for both methods in almost all cases. In some cases ML-CDN achieves a runtime reduction of factor 4 or 5 compared to CDN. This improvement comes from saving several iterations until it achieves a support size comparable to that of the true solution. This fact is reflected in the number of iterations of ML-CDN compared to those of CDN.
The reductions in runtime for ML-newGLMNET are more limited, as newGLMNET usually converges in a few iterations.
In particular, in the dataset {\tt rcv1}, the support of the solution concentrates 88\% of the non-zeros in the matrix $X$, and the multilevel acceleration is unable to save computations.
%In particular, on the dataset rcv1 the support of the solution concentrates 88\% of the non-zeros in the data, and the multilevel acceleration is unable to save computations.
Still, the number of iterations is reduced or is similar, while the runtime decreases by up to 45\%.

%\begin{table}[t]
%	\small
%	\caption{performance results for the original methods and their multilevel accelerated versions.}
%	\label{tb:Results}
%	\begin{center}
%		\begin{tabular}{lrrcrrrrr}
%			\hline %\\
%			Dataset &\multicolumn{1}{c}{($n$,$l$)}&\multicolumn{1}{c}{$C$} &\multicolumn{1}{c}{$\|\ve{w}^*\|_0$} &  \multicolumn{1}{c}{nGLM} & \multicolumn{1}{c}{ML-nGLM} & \multicolumn{1}{c}{CDN} & \multicolumn{1}{c}{ML-CDN} \\ \hline \hline  %\\
%			news20   &  1355191 ,  15997 &   64 &  2792 & 3.52\,(13)  & {\bf 1.87\,\,\,(7)} &  16.31(182)& 7.57\,\,\,\,(9)\\
%			rcv1     &    47236 , 541920 &    4 & 10893 & 37.89\,(13) & {\bf 36.43\,(14)}&  167.54\,\,\,(86)& 90.72\,(19)\\
%			webspam  & 16609143 , 280000 &   64 &  7914 & 122.2\,\,\,\,(8)  & \textbf{87.87\,\,\,\,(1)} &  2228.4\,\,\,(51)& 532.0\,\,\,\,(1)\\
%			epsilon  &     2000 , 400000 &  0.5 &  1106 & 196.0\,(13) & \textbf{162.0\,(13)}& 2933.6 (139)& 1501.1\,(35)\\
%			gisette  &     5000 ,   6000 & 0.25 &   554 & 1.44\,(10)  & \textbf{0.92\,\,\,\,(4)}  &  19.89\,\,(91)& 3.76\,\,\,\,(7)\\
%			\hline %\\
%		\end{tabular}
%	\end{center}
%\end{table}

% epsilon was 0.005 here for all runs
\begin{table}[t]
	\small
\centering
\begin{tabular}{|c|c|c|c|c|c|c|c|c|c|c|}
\hline
\hline
\mc{3}{|c|}{Problem parameters}&nGLM &ML-nGLM& CDN& ML-CDN \\ \hline
\mr{2}{Data}&\textbf{$n$}  &\textbf{$C$}& time & time  & time & time \\
            &\textbf{$(m)$}&$\|\ve{w}^*\|_0$  & (it) &  (it) &  (it) &  (it)  \\
\hline
\mr{2}{news20}   &1355191 &64   &3.52s & {\bf 1.87s} &  16.31s & 7.57s  \\
              &(15997) & 2792 &(13)  & (7)  &  (182)  & (9) \\
\hline
\mr{2}{rcv1} &    47236     & 4    &  37.89 & {\bf 36.43}&  167.54& 90.72\\
             &    (541920)  & 10893&  (13) & (14)&  (86)& (19)\\
\hline
\mr{2}{webspam}  & 16609143 &   64  & 122.2  & \textbf{87.87} &  2228.4& 532.0\\
			    & (280000) &   7914  & (8)  & \textbf{(1)} &  (51)& (1)\\
\hline
\mr{2}{epsilon}  &     2000 &  0.5 & 196.0 & \textbf{162.0}& 2933.6& 1501.1\\
                &      (400000) &  1106 & (13) & (13)& (139)& (35)\\
\hline
\mr{2}{gisette}  &     5000  & 0.25 & 1.44  & \textbf{0.92}  &  19.89& 3.76\\
                &   (6000) & 554  & (10) & (4)  &  (91)& (7)\\
\hline
\hline
\end{tabular}
\label{tb:Results}
\caption{Performance results for accelerating the solution of $l_1$-regularized logistic regression.}
\end{table}

\section{Conclusions}

In this work we present a multilevel framework for solving $l_1$ regularized sparse optimization
problems. To solve such problems efficiently,
we take advantage of the expected sparseness of the solution. A multilevel hierarchy of problems of similar type is created and
traversed in order to accelerate the optimization process. This framework is then applied for solving the sparse inverse covariance estimation and the $l_1$-regularized Logistic Regression problems.  The former is challenging especially for large-scale data sets, due to time and memory limitations. In this case, the multilevel framework enables an incremental construction of the solution in the number of non-zeros and avoids rather dense iterates. The biggest advantage of our framework is observed when the problem is hard to solve (lower regularization parameter, higher condition number).

\section{Acknowledgements} The authors would like to thank Ms. Aviva Herman for her technical contribution.

\appendix
\section{Supplementary material for BCD-IC}

\subsection{Computing the Linesearch Matrices}\label{sec:computingBi}
In this section we describe how to calculate the matrices $\mat{B}_i$ in \eqref{eq:logdet_blocks} efficiently, using the matrices $W_{I_j}$ and $W_{N_j}$ that are computed before the linesearch procedure (See Algorithm \ref{alg:BCDiteration}).
These matrices can be computed very efficiently using properties of the Schur complement, avoiding the computational burden of solving large linear systems involving the sub-matrix $\mat{A}_{22}$.

First, $\mat{B}_0$ is readily available by inverting a small $|I_j|\times |I_j|$ matrix.
Denoting $A^{-1} = W$, and following Schur complement for the partitioning $\eqref{eq:partitioning}$, we have
\begin{equation}\label{eq:computingBi1}
W_{11} = (A_{11} - A_{12}A_{22}^{-1}A_{21})^{-1}.
\end{equation}
Therefore, since the indices partitioned as `1' are those in $I_j$,
then
\begin{equation}\label{eq:B_0}
\mat{B}_0 = W_{11}^{-1}
\end{equation}
is available as part of $W_{I_j}$, with little effort of inverting the small $W_{11}$.
Second, we also have
\begin{equation}\label{eq:computingBi2}
W_{21} = -A_{22}^{-1}A_{21}W_{11}.
\end{equation}
Then, $\mat{B}_1$ is also available as $\mat{B}_1 = \mat{\Delta}_{11} + \mat{T} + \mat{T}^T$,
where $\mat{\Delta}$ denotes $\mat{\Delta}_j$ and
\begin{equation}\label{eq:T}
\mat{T} = -\mat{\Delta}_{21}^T\mat{A}_{22}^{-1}\mat{A}_{21} = \mat{\Delta}_{12}W_{21}\mat{B}_0.
\end{equation}
The latter is available since $W_{21}$ is again a part of $W_{I_j}$.

For $\mat{B}_2 = -\mat{\Delta}_{12}\mat{A}_{22}^{-1}\mat{\Delta}_{21}$, we only need $\mat{A}_{22}^{-1}$ in the block that correspond to $N_j$. That is because $\mat{\Delta}_{21}$, for example, is non-zero only in the rows that correspond to $N_j$.  Following Schur complement we have
\begin{equation}\label{eq:computingBi3}
W_{22} = \mat{A}^{-1}_{22} - \mat{A}^{-1}_{22}A_{21}W_{11}A_{12}\mat{A}^{-1}_{22},
\end{equation}
and after plugging in \eqref{eq:B_0}, \eqref{eq:computingBi2} and considering symmetry, we get
\begin{equation}\label{eq:A22}
\mat{A}_{22}^{-1} = W_{22} - W_{21}\mat{B}_0W_{21}^T.
\end{equation}
Now, we need the values of this matrix only at the block that corresponds to the columns and rows in $N_j$ (an $N_j\times N_j$ matrix). These, again are available from the computation of $W_{N_j}$. Given this matrix, we compute
\begin{equation}\label{eq:A22_detailed}
\mat{B}_2 = \mat{\Delta}_{12}^TW_{22} \mat{\Delta}_{21} + \left[\mat{\Delta}_{12}^TW_{21}\right]\mat{B}_0\left[W_{21}^T\mat{\Delta}_{21}\right],
\end{equation}
where the matrix $\mat{\Delta}_{12}^TW_{21}$ in brackets is computed also for \eqref{eq:T}.
\subsection{Reducing matrix inversions in BCD-IC for a given support}\label{sec:reducingGradient}
In this section we show how to efficiently calculate the rows of $W_{I_j} = (A^{-1})_{I_j}$ which are used in a BCD-IC update for a block $I_j\subset\{1,...,n\}$, restricted to a given sparse support $\C$. Define the ``$\C$-neighborhood'' of $I_j$ as
\begin{equation}\label{eq:extended_set_C}
N_j^\C = \left\{ i: \exists k \notin I_j : (i,k)\in \C \right\}.
\end{equation}
Assume that we calculate $W_{N_j^{\C}}$, the columns of $A^{-1}$ that are in $N_j^{\C}$. Now, the rows of $W_{I_j}$ that are needed for BCD-IC are those in $I_j \cup N_j^\C$ because these are the only rows of $W_{I_j}$ that are necessary for computing the objective and gradient of \eqref{eq:sub_Quad_MLE_l1_prior}. Since those $N_j^\C$ rows are available in $W_{N_j^{\C}}$ from symmetry, we only need the full block $W_{I_jI_j}$ to have all the needed rows.

Recall the partitioning \eqref{eq:partitioning}, where the blocks denoted by '1' correspond to $I_j$ and those denoted by '2' correspond to $\{1,...,n\}\setminus I_j$. Following the Schur complement property \eqref{eq:computingBi2}, we are able to compute the term
\begin{equation}\label{eq:reducingGradients1}
K = A_{12}W_{21} = -A_{12}A_{22}^{-1}A_{21}W_{11},
\end{equation}
using the previously computed $W_{N_j^\C}$. This is not immediate, as $W_{21}$ is a sub-matrix of $W_{I_j}$. However, following symmetry we have the rows of $W_{21}$ that correspond to indices in $\N_j^{\C}$ in $W_{N_j^{\C}}$, and because $A_{12}$ has non-zeros only in those columns, the other rows of $W_{21}$ are multiplied by 0 and are not necessary for \eqref{eq:reducingGradients1}.

Now, using the known matrix $K$, we compute $W_{11}$ without solving linear systems. By inverting \eqref{eq:computingBi1} and multiplying by $W_{11}$ from the right we get
\begin{equation}\label{eq:reducingGradients2}
I = A_{11}W_{11} + K,
\end{equation}
where $K$ is defined in \eqref{eq:reducingGradients1} and $I$ denotes the identity matrix of size $|I_j|\times |I_j|$. This helps us compute $W_{11} = A_{11}^{-1}(I - K)$ with little effort, assuming the block $I_j$ is small.

\bibliographystyle{siam}
\bibliography{patentprovisional}

\end{document}

%% file: SYNTH.tex
% epsilon = 0.05
\begin{tabular}{|c|c|c|c|c|c|c|c|c|c|c|c|c|c|}
\hline
\hline
\mc{2}{|c|}{Problem parameters}&ML-BCD&BCD cont.&DC-BCD &BCD-IC&BIG-QUIC\\\hline
$n$&$\lambda$&time&time&time&time&time\\
$(m)$&$\kappa({A}^*)$&(it)&(it)&(it)&(it)&(it)\\
$\|\Sigma^{-1}\|_0$&$\|A^*\|_0$&max-supp&max-supp&max-supp&max-supp&max-supp\\\hline\hline
     &0.70&\textbf{545s}&2024s&886s&1447s&3608s\\
     &4.19&(5)&(6)&(1)&(4)&(4)\\
     &491388&\textbf{492744}&497188&493076&501504&501528\\
\cline{2-7}
124294&0.65&\textbf{948s}&2802s&1841s&4322s&10467s\\
(200)&9.31&(6)&(6)&(2)&(8)&(7)\\
867224&908680&911034&940648&\textbf{910190}&980758&981412\\
\cline{2-7}
    &0.60&\textbf{2187s}&4668s&2751s&8420s&13704s\\
    &24.89&(7)&(7)&(2)&(10)&(7)\\
    &1505346&\textbf{1505346}&1621252&1505852&1853252&1858108\\
\cline{2-7}
    &0.55&\textbf{5006s}&8786s&6674s&15647s&20636s\\
    &68.61&(8)&(8)&(3)&(8)&(6)\\
    &2324416&\textbf{2324416}&2612854&2325800&3598916&3617662\\
\hline
\hline
&0.71&\textbf{2054s}&8237s&3547s&6093s&23233s\\
&4.17&(5)&(6)&(1)&(4)&(4)\\
&1040121&\textbf{1042325}&1051423&1043531&1060405&1060477\\
\cline{2-7}	
249045&0.66&\textbf{3909s}&11908s&5114s&13569s&57127s\\
(200)&9.50&(6)&(6)&(1)&(6)&(7)\\
1739435&1943239&1949805&2015077&\textbf{1945279}&2101425&2103019\\
\cline{2-7}			
&0.61&\textbf{7775s}&20068s&11656s&31293s&67756s\\
&26.65&(7)&(7)&(2)&(8)&(7)\\
&3273373&\textbf{3273373}&3536169&3280735&4068427&4082373\\
\cline{2-7}	
    &0.56&\textbf{18792s}&40898s&23189s&66352s&121707s\\
    &90.60&(8)&(8)&(2)&(7)&(7)\\
    &5136827&\textbf{5136827}&5812203&5155709&8112377&8191103\\
\hline
\hline	 		
&0.73&\textbf{5960s}&30774s&13310s&23024s&94976s\\
&2.96&(5)&(6)&(1)&(4)&(4)\\
&1899994&1905608&1915628&\textbf{1904166}&1929250&1929358\\
\cline{2-7}			 		 		
498604&0.69&\textbf{11958s}&43500s&17806s&39646s&181759s\\
(200)&7.32&(6)&(6)&(1)&(5)&(6)\\
3484746&3352522&3361596&3444112&\textbf{3358350}&3539656&3541236\\
\cline{2-7}			 		 		 		
&0.65&\textbf{24507s}&62334s&38014s&113950s&263444s\\
&18.53&(6)&(7)&(2)&(9)&(7)\\
&5310282&\textbf{5310282}&5647748&5348420&6156774&6168488\\
\cline{2-7}		
 &0.61&47647s&103961s&\textbf{40169s}&192403s&354082s\\
 &48.46&(7)&(7)&(1)&(9)&(7)\\
 &8170008&\textbf{8170008}&8923932&8183808&10930776&10990358\\
\hline
\hline
\end{tabular} 

%% file: GEO-ML.tex
% epsilon was 0.005 here for all runs
%\begin{footnotesize}
\begin{tabular}{|c|c|c|c|c|c|c|c|c|c|c|}
\hline
\hline
\mc{3}{|c|}{Problem parameters}&ML-BCD &BCDcont.& DC-BCD& BCD-IC & BIG-QUIC\\ \hline
\mr{2}{Data}&\textbf{$n$}  &\textbf{$\lambda$}& time & time  & time & time & time\\
            &\textbf{$(m)$}&\textbf{$\kappa(A^*)$}  & (it) &  (it) &  (it) &  (it)  & (it)\\
			%&              &\textbf{$\|A^*\|_0$} & $\max_k$\textbf{$\|A^{(k)}\|_0$}  & $\max_k$\textbf{$\|A^{(k)}\|_0$}  & $\max_k$\textbf{$\|A^{(k)}\|_0$}  & $\max_k$\textbf{$\|A^{(k)}\|_0$}& $\max_k$\textbf{$\|A^{(k)}\|_0$}  \\
			&              &\textbf{$\|A^*\|_0$}& max-supp  & max-supp  & max-supp  & max-supp& max-supp  \\
\hline
\hline
         &     &0.73  &\textbf{183s}&  217s   &  406s   &   429s     & 4328s\\
         &     &347.5  &(9)&              (7)    & (5)    & (7)    & (11)\\
GSE-      &     &237911 &\textbf{237911}&           240507    & 266567    &  499361 & 776831\\
\cline{3-8}
1898:     &21775 &0.70  &\textbf{243s} & 317s &  1457s & 1017s   & 6502s\\
Liver    &(182) & 530.0 &(9) &   (7)  &  (5)  & (7) & (13)\\
cancer    &      &294285&  \textbf{294291} &   320019 &   1376909  &  1148951 & 1265145\\
\cline{3-8}
        &      &0.67  &\textbf{491s} & 585s  & 1811s&  2380s & Not\\
          &      &766.3 &(11) &      (8)   &  (5) & (8) & converged\\
          &      &346999  &\textbf{348839} &    362069  & 1045079& 1747431 & 1955983\\
%GSE1898: &21775&0.67&346999&491s & 12& 348839&585s & 9  & 362069 &6 (1811s) & 6  & 1045079 &2380s & 9 & 1747431 \\
%Liver    &(182)&0.70&294285&243s & 10& 294291&317 s & 8  & 320019 &1457s & 6 & 1376909 &1017s & 8  & 1148951 \\
%Cancer&&0.73&237911&183s& 10& 237911&217s & 8 & 240507 &406s & 6  & 266567 &429s & 8 & 499361 \\
\hline
\hline
    &     &0.75  &\textbf{62s}&  107s    &    101s  &   133s  &  1184s \\
          &     & 177&(7)&                (7)     &     (4)  &   (7) & (6)    \\
GSE-          &     &108895&\textbf{108895}&   131497     &     109285  &   168465 &  189639 \\
\cline{3-8}
20194:    &22283&0.70  &\textbf{146s} &   237s & 263s &329s   & 3556s\\
Breast    &(278)&   544   &(8) &    (8)  & (4)   &  (7)   & (11)\\
cancer    &     &197809&\textbf{197809} & 219671  & 233963   &460665 & 559571  \\
\cline{3-8}
          &     &0.65  &\textbf{360s}&515s  &   1424s &   850s  &11068s \\
     &     & 1291 & (11)&                 (8)    &    (5)   &   (7) & (18)    \\
	&     &309665&\textbf{311401}&         325839    &  839419   &   718461 & 1454469  \\
%GSE20194:&22283&0.65&309665&360s& 12& 311401&515 s & 9  & 325839 &1424s & 6 & 839419 &850s & 8 & 718461 \\
%Breast&(278)&0.70&197809&146 s & 9& 197809&237s & 9  & 219671 &263s & 5 & 233963 &329s & 8 & 460665 \\
%Cancer&&0.75&108895&62s& 8& 108895&107s & 8  & 131497 &101s & 4  & 109285 &133s & 8 & 168465 \\
\hline
\hline
        &     &0.81  &\textbf{177s} &  494s   &   500s&  603s  &8001s   \\
        &     &220 &(6) &   (6)     &   (4) &    (6)  & (8)   \\
GSE- 	&     &328441&\textbf{329531} &  379859     &   350591 &496529 & 640097   \\
\cline{3-8}
17951:  &54675&0.78  &\textbf{494s} &  1258s  &   844s&    1347s     & 13807s\\
Prostate    &(154)&  381&(7) &                 (9)    &    (2) &  (6) & (10)    \\
cancer    & &538061&\textbf{538061} &  559951    &    557277 &  980713  & 1370869   \\
\cline{3-8}
     &     &     0.75  &\textbf{1484s} & 2125s &   3432s &   4567s   &  30850s  \\
     &     &        669   & (8) &                 (7)     &   (4) &    (6) &(14)    \\
	 &     &805883&\textbf{805883} &   823593     &   921863 &    1633517     & 2644767\\
%GSE17951:&54675&0.75&805883&1484s & 9& 805883&2125s & 8  & 823593 &3432s & 5  & 921863 &4567s & 7 & 1633517 \\
%Prostate&(154)&0.78&538061&494s & 8& 538061&1258s & 10 & 559951 &844s& 3  & 557277 &1347s & 7  & 980713 \\
%Cancer& &0.81&328441&177s & 7& 329531&494s & 7  & 379859 &500 s & 4  & 350591 &603s & 7  & 496529 \\
\hline
\hline
        &      &0.94   &\textbf{1472s}  & 3333s &   3353s &  6051s  & 67562s\\
          &      & 44 &(7)    &   (6) &     (3) &   (8)     & (8)\\
GSE-          &      &1976582&\textbf{1976582}    & 2052722 & 2007554    &2804832 & 3092670     \\
\cline{3-8}
14322:     &104702&0.92   &\textbf{4571s}  &7590s &     8308s &     16794s  & 56089s\\
Liver    &(76)  & 104  &(8)    &     (6) &       (3) &    (8) & (11)\\
cancer    &  &3394500&\textbf{3394500}  & 3514358 &       3403044 &  5402144 & 6170556    \\
\cline{3-8}
          &      &0.90   &\textbf{14456s} &19421s &   21802s &    65674s    & 155671s\\
         &      & 181&(10)   &           (7) &      (3) &   (9)   & (15)\\
         &      &4972394&\textbf{4972394}   &   5168408 &       4999572 &   11049268& 10390290    \\
%GSE14322:&104702&0.90&4972394&14456s & 11& 4972394&19421s & 8  & 5168408 &21802s & 4 & 4999572 &65674s & 10 & 11049268 \\
%Liver&(76)&0.92&3394500&4571s & 9& 3394500&7590s & 7  & 3514358 &8308s & 4  & 3403044 &16794s & 9 & 5402144 \\
%Cancer& &0.94&1976582&1472s & 8& 1976582&3333s & 7 & 2052722 &3353s & 4  & 2007554 &6051s & 9  & 2804832 \\
\hline
\hline
\end{tabular}
%\end{footnotesize} 